\documentclass[11pt,letterpaper,notitlepage,twoside]{article}

\usepackage{titling}
\usepackage{enumitem}

\usepackage[top=4cm, bottom=3.5cm, left=3.8cm, right=3.8cm, columnsep=20pt]{geometry}
\usepackage{indentfirst}
\usepackage[dvipsnames]{xcolor}
\definecolor{FrenchPink}{RGB}{255,118,164}
\definecolor{VividMulberry}{RGB}{192,17,215}
\definecolor{DenimBlue}{RGB}{47,60,190}
\usepackage{hyperref}
\hypersetup{
    colorlinks=true,     
    linkcolor=DenimBlue,    
    citecolor=FrenchPink,     
    filecolor=VividMulberry,      
    urlcolor=VividMulberry,       
}
\usepackage{amsmath}
\usepackage{amsthm}
\usepackage{amssymb}
\usepackage{mathrsfs}
\usepackage{mathtools}
\usepackage{etoolbox}
\usepackage{dsfont}
\usepackage{tikz-cd}

\usepackage{tabularx}
\usepackage{graphicx}
\usepackage{caption}
\usepackage{subcaption}
\usepackage{transparent}
\usepackage{float}
\usepackage{booktabs}
\usepackage{wrapfig}
\usepackage[final]{pdfpages}

\newtheoremstyle{thm}{}{}{\slshape}{}{\bfseries}{}{.5em}{}
\theoremstyle{thm}
\newtheorem{thm}{Theorem}
\newtheorem{prop}{Proposition}
\newtheorem{lem}{Lemma}
\newtheorem{cor}{Corollary}
\newtheorem*{conj*}{Conjecture}
\newtheorem*{defn*}{Definition}

\newtheorem{thmout}{Theorem} 

\newtheoremstyle{ex}{}{}{}{}{\scshape}{{:}}{.5em}{}
\theoremstyle{ex}
\newtheorem*{ex*}{Example}
\newtheorem*{ack*}{Acknowledgements}

\newtheoremstyle{rem}{}{}{}{}{\scshape}{}{.5em}{}
\newtheorem{rem}{Remark}

\newtheoremstyle{pr}{}{}{}{}{\scshape}{{:}}{.5em}{}
\theoremstyle{pr}
\newtheorem*{pr*}{Proof}
\AtEndEnvironment{pr*}{\null\hfill\qedsymbol}

\renewenvironment{proof}[1][\proofname]{{\noindent\scshape #1:}}{\null\hfill\qedsymbol}

\usepackage{lipsum}

\usepackage{titlesec}
\renewcommand{\thesection}{\Roman{section}}
\titleformat{\section}[block]{\large\scshape}{\thesection.}{0.5em}{}
\renewcommand{\thesubsection}{\Roman{section}.\alph{subsection}}
\titleformat{\subsection}[block]{\scshape}{\thesubsection.}{0.5em}{}
\titleformat{\subsubsection}[block]{\bfseries}{}{0.5em}{}

\newcommand{\titre}{Convergence and Riemannian bounds} 
\newcommand{\titrep}{on Lagrangian submanifolds} 
\newcommand{\titrepp}{Convergence and Riemannian bounds on Lagrangian submanifolds}
\newcommand{\prenomauteur}{Jean-Philippe} 
\newcommand{\nomauteur}{Chass\'e} 

\newcommand{\pagetitre}[4]{
\noindent\rule{\linewidth}{1pt}
\begin{center}
\begin{tabular}{c}
{\LARGE\textmd{\scshape #1}}\\[2pt]
{\LARGE\textmd{\scshape #2}}\\[15pt]
{\large #3 {\scshape #4}$^*$}
\end{tabular}
\end{center}
\vspace{-0.2cm}
\noindent\rule{\linewidth}{1pt}}

\usepackage{fancyhdr}
\fancypagestyle{first}{
\pagestyle{fancy}
\fancyhead[C]{}
\fancyhead[RO,LE]{}
\fancyhead[LO,RE]{}
\fancyfoot[C]{}
\fancyfoot[L]{{\footnotesize Latest update: \today \\[1pt] $^*$The author was supported by a NSERC doctoral research scholarship.}}

}

\makeatletter
\renewenvironment{abstract}{%
    \if@twocolumn
      \section*{\abstractname}%
    \else\small 
      \begin{center}%
        {\scshape \abstractname\vspace{-0.25cm}}
      \end{center}%
      \quotation
    \fi}
    {\if@twocolumn\else\endquotation\fi}
\makeatother

\makeatletter
\renewcommand\tableofcontents{%
	\begin{center}%
	{\scshape \contentsname\vspace{-0.25cm}}
	\end{center}%
	\@mkboth{\MakeUppercase\contentsname}{\MakeUppercase\contentsname}%
	\@starttoc{toc}%
}
\makeatother

\newcommand{\Z}{\mathds{Z}}

\newcommand{\R}{\mathds{R}}

\newcommand{\D}{\mathds{D}}
\newcommand{\Id}{\mathds{1}}

\newcommand{\del}{\partial}

\newcommand{\grad}{\operatorname{grad}}

\let\phi\varphi
\let\epsilon\varepsilon
\let\emptyset\varnothing

\begin{document}
\thispagestyle{first}

\pagetitre{\titre}{\titrep}{\prenomauteur}{\nomauteur}

\begin{abstract}
\noindent
We consider collections of Lagrangian submanifolds of a given symplectic manifold which respect uniform bounds of curvature type coming from an auxiliary Riemannian metric. We prove that, for a large class of metrics on these collections, convergence to an embedded Lagrangian submanifold implies convergence to it in the Hausdorff metric. This class of metrics includes well-known metrics such as the Lagrangian Hofer metric, the spectral norm and the shadow metrics introduced by Biran, Cornea and Shelukhin \cite{BiranCorneaShelukhin2018}. The proof relies on a version of the monotonicity lemma, applied on a carefully-chosen metric ball.
\end{abstract}
\noindent\rule{\linewidth}{1pt}
\tableofcontents
\noindent\rule{\linewidth}{1pt}

\section{Introduction and main result} \label{sec:intro}
Spaces of Lagrangian submanifolds are generally analyzed in a metric sense at large scale. For example, there has been a great amount of work on whether the space of Lagrangian submanifolds, subject to some topological constraints, has infinite diameter, or whether the metric admits upper bounds in terms of intersection numbers (see for example \cite{Khanevsky2009, Humiliere2012, Zapolsky2013, Seyfaddini2014, BiranCorneaShelukhin2018}). However, a less studied problem is the local behavior of those metrics, even in the case of the well-known Lagrangian Hofer metric. This is probably due to the fact that, without additional constraints, converging sequences can be quite wild from a set-theoretic standpoint. Furthermore, before the advent of shadow metrics, one only looked at Hamiltonian isotopic Lagrangian submanifolds ---~or at least, conjecturaly Hamiltonian isotopic ones. \par

The purpose of this paper is to show that, when we only look at Lagrangian submanifolds behaving well with respect to some auxiliary Riemannian metric, converging sequences cannot be wild. More precisely, we will show that they also converge to the same Lagrangian submanifold in the Hausdorff metric associated to the auxiliary Riemannian metric. Note however that we only look at sequences converging to an embedded Lagrangian submanifold. This restriction is necessary since sequences converging in certain metrics could theoretically converge to an immersed Lagrangian submanifold. As noted above, this result is of particular interest when applied to the weighted fragmentation metrics of Biran, Cornea, Shelukhin and Zhang \cite{BiranCorneaShelukhin2018, BiranCorneaZhang2021}. Indeed, these metrics exist on spaces of Lagrangian submanifolds which are not necessarily of the same homotopy type, let alone Hamiltonian isotopic. But, as we shall see below, the presence of bounds coming from a Riemannian metric forces neighboring Lagrangian submanifolds to be homeomorphic. \par

\subsection{Some notation and definitions} \label{subsec:notation}
Before writing down the main result in a more precise form, we need to fix some notation and definitions. \par

\subsubsection*{Riemannian bounds} \label{subsubsec:metric}
We begin by introducing some Riemannian quantities that will serve to restrict the classes of the Lagrangian submanifolds that we will consider. \par

Let $L$ be a submanifold of a Riemannian manifold $(M,g)$. We can see its second fundamental form $B_L$ as a section of the bundle $(TL\otimes TL\otimes TL^\perp)^*\to L$, where $\perp$ denotes the orthogonal complement with respect to $g$. We thus define the \emph{norm} of the second fundamental form to be
\begin{align*}
||B_L||:=\sup_{x\in L} |B_L(x)|_{\tilde{g}_x},
\end{align*}
where $\tilde{g}_x$ is the scalar product induced by $g$ on $(T_xL\otimes T_xL\otimes T_xL^\perp)^*$. When $\dim L=1$ and $\dim M=2$ ---~a case that will be of particular interest to us later on~--- this is just the supremum of the geodesic curvature of $L$. \par

In general, uniformly bounding the norm of the second fundamental form will not be enough for our purposes. We will thus make use of another quantity, which gives a better control on the way $L$ is embedded in $M$. \par

\begin{defn*}
Let $(M,g)$ be a Riemannian manifold, and $L$ be a submanifold. Let $\epsilon\in(0,1]$. We say that $L$ is \textbf{$\epsilon$-tame} if
\begin{align*}
\frac{d_M(x,y)}{\min\{1,d_L(x,y)\}}\geq\epsilon \qquad\forall x\neq y\in L,
\end{align*}
where $d_M$ is the distance function on $M$ induced by $g$, and $d_L$ is the distance function on $L$ induced by $g|_L$.
\end{defn*}

\begin{rem} \label{rem:names}
This is a variation of the tame condition appearing in work of Sikorav \cite{Sikorav1994}. More precisely, it is equivalent to the (T'1) condition. This condition also appeared under the name $\epsilon$-Lipschitz in work of Groman and Solomon \cite{GromanSolomon2014, GromanSolomon2016}.
\end{rem}

\subsubsection*{Collections of Lagrangian submanifolds} \label{subsubsec:lagrangians}
In general, there is no hope of being able to meaningfully compare two arbitrary Lagrangian submanifolds of a given symplectic manifold. That is why, when defining metrics on spaces of Lagrangian submanifolds, it is necessary to consider more restricted collections of Lagrangian submanifolds, e.g.\ Hamiltonian isotopic ones, exact ones, etc. \par

In the present paper, the symplectic properties of the Lagrangian submanifolds considered will not matter inasmuch as they allow for the definition of well-behaved metrics between them. However, to give an idea of the collections to which our result applies, we present some interesting choices of collection $\mathscr{L}^\star(M)$ of Lagrangian submanifolds of a given symplectic manifold $(M,\omega)$. Note that throughout this paper, we will assume that $M$ is connected, and either closed or convex at infinity. We will also assume that the Lagrangian submanifolds are closed and connected. \par

\begin{itemize}[leftmargin=2.8cm]
\item[($\star=L_0$):] Here, $L_0$ is some fixed Lagrangian submanifold. Then, $\mathscr{L}^{L_0}(M)$ is the Hamiltonian orbit of $L_0$, i.e.\  the set of Lagrangian submanifolds of the form $\phi(L_0)$ for some Hamiltonian diffeomorphism with compact support $\phi$. \par

\item[($\star=e$):] For this collection to make sense, we need to suppose that $M$ is exact. Then, $\mathscr{L}^{e}(M)$ is the collection of exact Lagrangian submanifolds. \par

\item[($\star=we$):] Here, $\mathscr{L}^{we}(M)$ is the collection of weakly exact Lagrangian submanifolds, i.e.\ Lagrangian submanifolds $L$ such that the morphism $\omega:\pi_2(M,L)\to\R$ given by integration with respect to $\omega$ is trivial. Note that the existence of such a Lagrangian submanifold implies that $M$ is symplectically aspherical. \par

\item[($\star=m(\rho,d))$:] Here, $\rho>0$ and $d\in\Z_2$. Then, $\mathscr{L}^{m(\rho,d)}(M)$ is the collection of Lagrangian submanifolds $L$ such that
\begin{enumerate}[label=(\roman*)]
\item the Maslov index $\mu:\pi_2(M,L)\to \Z$ satifies $\omega =\rho \mu$;
\item the minimum Maslov number is $N_L\geq 2$;
\item the modulo 2 count of $J$-holomorphic disks with boundary along $L$ of Maslov index 2 ---~for $J$ generic~--- is equal to $d$.
\end{enumerate}
\end{itemize}

As noted by Biran, Cornea and Zhang \cite{BiranCorneaZhang2021}, when one is interested in fragmentation metrics, it might be necessary to restrict oneself to a subcollection of one of the above choices. \par

The result that we will present does not hold for all Lagrangian submanifolds in $\mathscr{L}^\star(M)$. Indeed, we will need to impose some uniform bounds coming from Riemannian geometry. Therefore, we fix a Riemannian metric $g$ on $M$ and constants $\Lambda\geq 0$, $\epsilon\in (0,1]$. We then introduce two new types of subcollections:
\begin{gather*}
\mathscr{L}^\star_\Lambda(M,g):=\left\{L\in\mathscr{L}^\star(M)\ \middle|\ ||B_L||\leq\Lambda\right\} \\
\mathscr{L}^\star_{\Lambda,\epsilon}(M,g):=\left\{L\in\mathscr{L}^\star_\Lambda(M,g)\ \middle|\ \text{$L$ is $\epsilon$-tame}\right\}. \\
\end{gather*}
We recall that we always consider our Lagrangian submanifolds to be closed and connected. When it is evident from the context, we will omit $g$ from the notation. \par

\subsubsection*{\texorpdfstring{$J$}{J}-adapted metrics on collections of Lagrangian submanifolds} \label{subsubsec:fragmentation}
We now describe the type of metrics that we will be putting on our collections $\mathscr{L}^\star(M)$. In order to do so, we fix an almost complex structure $J$ that is compatible with $\omega$. We begin by presenting certain pseudometrics which we will call \emph{$J$-adapted}. These are defined using an auxiliary family $\mathscr{F}\subseteq\mathscr{L}^\star(M)$. We thus also fix such a family. For $J$ generic, a $J$-adapted pseudometric $d^\mathscr{F}$ will be one of the following. \par

\begin{itemize}[leftmargin=3.2cm]
\item[($d^\mathscr{F}=d_H$):] This is the case of the Lagrangian Hofer metric. It is then understood that $\mathscr{L}^\star(M)\subseteq \mathscr{L}^{L_0}(M)$ for some $L_0$, and that $\mathscr{F}=\emptyset$. Note that this is actually a metric.  \par

\item[($d^\mathscr{F}=\gamma$):] This is the case of the spectral norm. It is then understood that $\mathscr{L}^\star(M)\subseteq \mathscr{L}^{L_0}(M)$ for some $L_0\in \mathscr{L}^{we}(M)\cap\mathscr{L}^{m(1,0)}(M)$. In this case, we take $\mathscr{F}=\emptyset$. This is again a metric. Finally, when $M=T^*L_0$ and $L_0$ is spin, then the metric may be extended to be on $\mathscr{L}^{e}(M)\cap \mathscr{L}^{m(1,0)}(M)$ by work from Fukaya, Seidel and Smith \cite{FukayaSeidelSmith2008i, FukayaSeidelSmith2008ii}, and from Nadler \cite{Nadler2009}. \par

\item[($d^\mathscr{F}=\gamma_\mathrm{ext}$):] This is a variant of the usual spectral norm, as defined in \cite{KislevShelukhin2018}, where it is also shown that it is a metric. We also have that $\mathscr{L}^\star(M)\subseteq \mathscr{L}^{L_0}(M)$ and $\mathscr{F}=\mathscr{F'}=\emptyset$, but we only ask that $L_0\in \mathscr{L}^{we}(M)$. However, $M$ has then to be closed and monotone, i.e.\ the diagonal of $(M\times M,\omega\oplus -\omega)$ is in $\mathscr{L}^{m(\rho,d)}(M\times M)$. \par

\item[($d^\mathscr{F}=d^\mathscr{F}_\mathcal{S}$):] These are the shadow pseudometrics appearing in work of Biran, Cornea and Shelukhin \cite{CorneaShelukhin2019, BiranCorneaShelukhin2018}. It is then understood that either $\mathscr{L}^\star(M)\subseteq \mathscr{L}^{e}(M)$, $\mathscr{L}^\star(M)\subseteq \mathscr{L}^{we}(M)$ or that $\mathscr{L}^\star(M)\subseteq \mathscr{L}^{m(\rho,d)}(M)$ for some $\rho>0$ and $d\in\Z_2$. Note that these will in general be degenerate. \par

\item[($d^\mathscr{F}=s^\mathscr{F}_\mathrm{alg}$):] These are the so-called algebraic fragmentation pseudometrics also appearing in work of Biran, Cornea and Shelukhin \cite{BiranCorneaShelukhin2018}. As above, it is then understood that $\mathscr{L}^\star(M)$ is in either $\mathscr{L}^{e}(M)$, $\mathscr{L}^{we}(M)$ or some $\mathscr{L}^{m(\rho,d)}(M)$. They might also be degenerated. \par

\item[($d^\mathscr{F}=D^\mathscr{F}$):] There are possibly many other weighted fragmentation pseudometrics ---~as defined by Biran, Cornea and Zhang \cite{BiranCorneaZhang2021}~--- that belong to this class. \par
\end{itemize}

We use here the words pseudometric and metric in the less restrictive sense where it may take infinite values.  \par

In all of those cases, the $J$-adapted property tying these pseudometrics to the almost complex structure $J$ is the following. For any $L,L'\in \mathscr{L}^\star(M)$ intersecting transversally and any point $x\in L$, there exists a $J$-holomorphic polygon $u$ passing through $x$ such that
\begin{align*}
\omega(u)\leq 2d^\mathscr{F}(L,L').
\end{align*}
Furthermore, $u$ has boundary along $L$, $L'$, and along Lagrangian submanifolds in $\mathscr{F}$ (see \cite{Seidel2008} for a more detailed exposition on $J$-holomorphic polygons and on their role in symplectic topology). If the Lagrangian submanifolds above do not intersect transversally, then they should be replaced by arbitrarily small Hamiltonian perturbation. A more precise definition will be given in Subsection \ref{subsec:j-adapted}. \par

As noted above, weighted fragmentation pseudometrics such as $d^\mathscr{F}_\mathcal{S}$ and $s^\mathscr{F}_\mathrm{alg}$ are in general degenerate. Intuitively, this is because $D^\mathscr{F}(L,L')$ measures the size of a cone decomposition of $L$ in terms of $L'$ and elements of $\mathscr{F}$ in the appropriate derived Fukaya category $\operatorname{DFuk}^\star(M)$. Therefore, whenever $L,L'\in\mathscr{F}$, we have that $D^\mathscr{F}(L,L')=0$. As noted by Biran, Cornea and Shelukhin \cite{BiranCorneaShelukhin2018}, a way to get around this issue is to use another auxiliary family $\mathscr{F'}$. The family must be so that
\begin{align} \label{eqn:intersection}
\left(\overline{\bigcup_{F\in\mathscr{F}}F}\right)\bigcap\left(\overline{\bigcup_{F'\in\mathscr{F'}}F'}\right)
\end{align}
is disconnected enough, e.g.\ discrete or totally disconnected. Indeed, under the later connectivity assumption, $\tilde{D}^{\mathscr{F},\mathscr{F'}}:=\max\{D^\mathscr{F},D^\mathscr{F'}\}$ is a true metric in the shadow and algebraic fragmentation cases. \par

Based on this phenomenon, for any families $\mathscr{F}$ and $\mathscr{F'}$ such that the intersection (\ref{eqn:intersection}) is discrete, we will call $\hat{d}^{\mathscr{F},\mathscr{F'}}:=\max\{d^\mathscr{F},d^\mathscr{F'}\}$ a \emph{$J$-adapted metric}. As we will see soon, the name is warranted: $\hat{d}^{\mathscr{F},\mathscr{F'}}$ will indeed be a metric on the appropriate space. \par

\subsection{A conjecture on convergence in Lagrangian spaces} \label{subsec:conjecture}
We now introduce a conjecture due to Cornea and explain how we will show it holds under some additional assumptions in high dimensions and without them in dimension 2. \par

We fix a connected symplectic manifold $(M,\omega)$ which is closed or convex at infinity. We also fix an $\omega$-compatible almost complex structure $J$ such that $g_J$ has uniformly bounded sectional curvature and with injectivity radius uniformly bounded away from zero. We also require that $(M,g_J)$ be a complete Riemannian manifold. Note that a symplectic manifold that is convex at infinity always admits such a $J$. This has been proven for the case when $M$ is a twisted cotangent bundle by Cieliebak, Ginzburg and Kerman \cite{CieliebakGinzburgKerman2004}, following a suggestion of Sikorav. As noted there, the proof easily adapts to the case when $M$ is instead convex at infinity. \par

We also fix a collection of Lagrangian submanifold $\mathscr{L}^\star(M)$ and a $J$-adapted metric $\hat{d}^{\mathscr{F},\mathscr{F'}}$ on it. By the examples in the previous subsection, we have multiple choices of appropriate collections $\mathscr{L}^\star(M)$ and metrics $\hat{d}^{\mathscr{F},\mathscr{F'}}$. \par

\begin{conj*}[Cornea, 2018]
The topology induced by the Hausdorff metric on $\mathscr{L}^\star_{\Lambda}(M)$ is stronger than the one induced by the $J$-adapted metric $\hat{d}^{\mathscr{F},\mathscr{F}'}$. In other words, if $L_n\to L_0$ in $\hat{d}^{\mathscr{F},\mathscr{F}'}$, then $L_n\to L_0$ in the Hausdorff metric $\delta_H$ induced by $g_J$.
\end{conj*}

\begin{rem} \label{rem:generalization}
The conjecture was originally stated for the weighted fragmentation metrics appearing in \cite{BiranCorneaShelukhin2018}. However, the proof lends itself naturally to a generalization to a larger class of metrics behaving well with respect to $J$-holomorphic curves.
\end{rem}

Note that the conjecture implies that, for $n$ large, $L_n$ is homeomorphic to $L_0$. Indeed,  we have the inequality
\begin{align*}
\delta_{GH}(L_0,L_n)\leq \delta_H(L_0,L_n),
\end{align*}
where $\delta_{GH}$ denotes the Gromov-Hausdorff pseudometric on the collection of compact metric spaces (see \cite{BuragoBuragoIvanov2001} for a thorough exposition of the subject). Therefore, $\{L_n\}_{n\geq 1}$ also converges to $L_0$ in the Gromov-Hausdorff pseudometric. However, since $M$ has bounded sectional curvature and $\{L_n\}$ has uniformly bounded second fundamental form, $\{L_n\}$ has uniformly bounded sectional curvature by Gauss' equation. In particular, $\{L_n\}$ is a sequence of Alexandrov spaces with curvature uniformly bounded from below. Thus, by Perelman's stability theorem \cite{Perelman1991}, $L_n$ is homeomorphic to $L_0$ for $n$ large. \par 

Actually, we expect $L_n$ to be Hamiltonian isotopic to $L_0$ for $n$ large, and we know that to be true in some cases. Indeed, since $\{L_n\}_{n\geq 1}$ converges to $L_0$ in the Hausdorff metric, $L_n$ is eventually in a Weinstein neighborhood of $L_0$. Therefore, if the nearby Lagrangian conjecture holds for $L_0$, and if each $L_n$ is exact in the Weinstein neighborhood, e.g.\ it is exact in $M$ or simply connected, then $L_n$ has to be Hamiltonian isotopic to $L_0$. In particular, this is the case when $\{L_n\}\subseteq \mathscr{L}^e_{\Lambda}(M)$ and $\dim M=2$. \par

The main purpose of this paper is to prove the conjecture under the slightly stronger hypothesis that the Lagrangian submanifolds are also $\epsilon$-tame. \par

\begin{thmout} \label{thm:main}
Let $\{L_n\}_{n\geq 1}\subseteq\mathscr{L}^{\star}_{\Lambda,\epsilon}(M)$ be such that $L_n\to L_0\in\mathscr{L}^\star(M)$ with respect to a $J$-adapted metric $\hat{d}^{\mathscr{F},\mathscr{F}'}$. Then, $L_n\to L_0$ in the Hausdorff metric $\delta_H$ induced by $g_J$. \par

Moreover, when $\dim M=2$, the statement where $\mathscr{L}^{\star}_{\Lambda,\epsilon}(M)$ is replaced by $\mathscr{L}^{\star}_{\Lambda}(M)$ holds, i.e.\ the conjecture holds as initially stated in dimension 2.
\end{thmout}

The idea of the proof is as follows:
\begin{enumerate}[label=(\roman*)]
\item Since $\hat{d}^{\mathscr{F},\mathscr{F'}}$ is a $J$-adapted metric, for any $x\in L_0-(L_n\cup(\cup F))$ and $x'\in L_n-(L_0\cup(\cup F))$, there exist $J$-holomorphic polygons $u$ and $u'$ passing through $x$ and $x'$, respectively ---~modulo arbitrarily small perturbations. Furthermore, their area is bounded from above by $2\hat{d}^{\mathscr{F},\mathscr{F'}}(L_n,L_0)$. \par

\item By finding appropriate metric balls in $M$ and using a version of the monotonicity lemma, it is possible to find a lower bound for the area of these polygons for $n$ large. This bound depends only on $M$, $\Lambda$, $\epsilon$, and the distances $d_M(x,L_n\cup(\cup F))$, $d_M(x,L_n\cup(\cup F'))$, $d_M(x',L_0\cup(\cup F))$ and $d_M(x',L_0\cup(\cup F))$. By the previous step, this turns into a lower bound of $\hat{d}^{\mathscr{F},\mathscr{F'}}(L_n,L_0)$. \par

\item Using the fact that $(\overline{\cup F})\cap(\overline{\cup F'})$ is discrete, it is possible to turn the dependence on the different distances onto one on the Hausdorff distance $\delta_H(L_n,L_0)$. \par

\item The fact that $\hat{d}^{\mathscr{F},\mathscr{F'}}(L_n,L_0)\to 0$ then forces that $\delta_H(L_n,L_0)\to 0$.
\end{enumerate}

We can get rid of the dependence on $\epsilon$ when $\dim M=2$ because, in that case, we can make a better choice of metric balls on which the monotonicity lemma is applied. \par

\subsection{Structure of the paper} \label{subsec:structure}
The remainder of the paper is almost entirely dedicated to the proof of Theorem \ref{thm:main}. \par

In Section \ref{sec:proof-main}, we give a proper definition of $J$-adapted metrics. We then give the proof of the main theorem without any restriction on the dimension of the symplectic manifold $M$. This is done in two steps: proving an appropriate version of the monotonicity lemma, and showing that the existence of an appropriate $J$-holomorphic polygon implies the result. \par

In Section \ref{sec:proof-dim2}, we explain how to modify the proof of Section \ref{sec:proof-main} to get rid of $\epsilon$ when $\dim M=2$. Essentially, it suffices to modify the ball on which we apply the monotonicity lemma. To find such a ball, we develop some combinatoric arguments for curves on surfaces. This section ends with some proofs of results from Riemannian geometry that we have not found explicited in the literature. \par

In Section \ref{sec:bad}, we end the paper with an example of a sequence of Lagrangian submanifolds that do not respect uniform Riemannian bounds and do not respect the conclusion of Theorem \ref{thm:main}. This shows the sufficiency of said bounds. \par

\begin{ack*}
This research is part of my PhD thesis and was financed by my NSERC scholarship. I could not thank enough my advisor, Octav Cornea, for the many fruitful discussions we've had throughout the research and writing process which led to this paper. I would also like to thank Egor Shelukhin and Jordan Payette for their many helpful comments and insights, especially regarding the two-dimensional proof of Theorem \ref{thm:main}.
\end{ack*}

\section{The general proof} \label{sec:proof-main}
In this section, we give the proof of Theorem \ref{thm:main} without restrictions on the dimension of $M$. Actually, we prove a more precise estimate on the relation between $\hat{d}^{\mathscr{F},\mathscr{F'}}$ and $\delta_H$. In what follows, the almost complex structure $J$ is fixed. We also fix constants $K_0,r_0>0$ such that the sectional curvature of $M$ respects $|K|\leq K_0$, and its injectivity radius respects $r_\mathrm{inj}(M)\geq r_0$. We always assume that $(M,g_J)$ is complete. \par

\begin{thm} \label{thm:main-precise}
For any $J$-adapted metric $\hat{d}^{\mathscr{F},\mathscr{F'}}$ on $\mathscr{L}^\star_{\Lambda,\epsilon}(M,g_J)$, there exist constants $R=R(K_0,r_0,\Lambda,\epsilon,(\overline{\cup F})\cap (\overline{\cup F'}))>0$ and $C=C(K_0,r_0,\epsilon)>0$ such that whenever $\hat{d}^{\mathscr{F},\mathscr{F'}}(L,L')<R$, then
\begin{align*}
\hat{d}^{\mathscr{F},\mathscr{F'}}(L,L')\geq C\delta_H(L,L')^2.
\end{align*}
\end{thm}

Clearly, the first part of Theorem \ref{thm:main} follows directly from Theorem \ref{thm:main-precise}. \par

\begin{rem} \label{rem:groman-solomon}
As explained above, the proof relies on a version of the monotonicity lemma. However, it was pointed out to us by Shelukhin that a modification of Groman--Solomon's reverse isoperimetric inequality \cite{GromanSolomon2014} would also yield a proof of Theorem \ref{thm:main}. Such a modification of the inequality has turned out to be much more difficult to prove than what ended up being presented here. Furthermore, additional Riemannian bounds seem then to be required, e.g.\ $C^1$ bounds on $B_L$. On the other hand, this would have the advantage of giving a linear inequality in Theorem \ref{thm:main-precise}, instead of a quadratic one. \par
\end{rem}

\subsection{\texorpdfstring{$J$}{J}-adapted metrics} \label{subsec:j-adapted}
Before going in the proof of Theorem \ref{thm:main-precise}, we need to give a precise definition of what we mean by $J$-adapted. We also explain how the metrics enumerated in the introduction fit in this definition. \par

We begin by clarifying what we mean by a $J$-holomorphic polygon, essentially following Seidel's book \cite{Seidel2008}. \par

\begin{defn*}
Let $(M,\omega)$ be a symplectic manifold equipped with a $\omega$-compatible almost complex structure $J$. Let $L_0,\dots, L_k$ be pairwise transverse Lagrangian submanifolds. Denote by $S_r$ the closed unit disk with $|r|\leq k+1$ punctures at its boundary. We equip $S_r$ with the standard complex structure $j$ and area form $\sigma$. We label the components of $\del S_r$ counterclockwise from $C_0$ to $C_k$, and the puncture from $\zeta_0$ to $\zeta_k$ accordingly. A \textbf{$J$-holomorphic polygon with boundary along $L_0,\dots,L_k$} is a smooth map $u:S_r\to M$ such that
\begin{enumerate}[label=(\roman*)]
\item $u$ is $(j,J)$-holomorphic;
\item $E(u):=\int_{S_r}|du|^2_{g_J}\sigma<\infty$;
\item $u(C_i)\subseteq L_i$ for all $i\in\{0,\dots,|r|-1\}$;
\item near $\zeta_i$, there are conformal coordinates $(s,t)\in [0,\infty)\times [0,1]$ such that $\lim_{s\to\infty} u(s,t)=:p_i\in L_{i-1}\cap L_i$.
\end{enumerate}
Furthermore, for some $x\in L_0-(\cup_{i=1}^k L_i)$, we will say that $u$ \textbf{passes through $x$} if $x\in u(C_0)$.
\end{defn*}

This allows us to define precisely what are $J$-adapted metrics. \par

\begin{defn*}
Let $\mathscr{L}^\star(M)$ be a collection of Lagrangian submanifolds on $M$. Let $\mathscr{F}\subseteq \mathscr{L}^\star(M)$. We say that a pseudometric $d^\mathscr{F}$ on $\mathscr{L}^\star(M)$ is \textbf{$J$-adapted} if for all $\delta>0$ and all $L, L'\in\mathscr{L}^\star(M)$, there exist Lagrangian submanifolds $F_1,\dots,F_k\in\mathscr{F}$ with the following property. \par

For any $C^0$- and Hofer-small Hamiltonian perturbations $\widetilde{L},\widetilde{L}',\widetilde{F}_1,\dots,\widetilde{F}_k$ of the Lagrangian submanifolds above making them pairwise transverse, and for any $x\in \widetilde{L}\cup\widetilde{L}'$, there exists a nonconstant $J$-holomorphic polygon $u:S_r\to M$ such that
\begin{enumerate}[label=(\roman*)]
\item has boundary along $\widetilde{L}$, $\widetilde{L}'$ and $\widetilde{F}_1,\dots,\widetilde{F}_k$; \par
\item passes through $x$; \par
\item respects the bound
\begin{align*}
\omega(u)\leq d^\mathscr{F}(L,L')+\delta.
\end{align*}
\end{enumerate}

Let $\mathscr{F'}\subseteq \mathscr{L}^\star(M)$ be such that
\begin{align*}
\left(\overline{\bigcup_{F\in\mathscr{F}}F}\right)\bigcap \left(\overline{\bigcup_{F'\in\mathscr{F'}}F'}\right)
\end{align*}
is discrete. We will call $\hat{d}^{\mathscr{F},\mathscr{F}'}:=\max\{d^\mathscr{F},d^\mathscr{F'}\}$ a \textbf{$J$-adapted metric} if $d^\mathscr{F}$ and $d^{\mathscr{F}'}$ are both $J$-adapted pseudometrics.
\end{defn*}

We conclude this subsection by explaining how the metrics mentioned above are indeed $J$-adapted metrics, at least when $J$ lies in some residual set. In the case when $\hat{d}^{\mathscr{F},\mathscr{F'}}$ is either a shadow metric $\hat{d}_\mathcal{S}^{\mathscr{F},\mathscr{F'}}$ or an algebraic one $\hat{s}_\mathrm{alg}^{\mathscr{F},\mathscr{F'}}$, this is proven in the course of Theorem 5.0.2 of \cite{BiranCorneaShelukhin2018}. By the inequality $\hat{d}_\mathcal{S}^{\emptyset,\emptyset}\leq d_H$, this also implies the result for the Lagrangian Hofer metric. \par

In fact, in the case of the Lagrangian Hofer metric, when the Lagrangian submanifolds involved are weakly exact, it is possible to get a sharper result. In that case, $d_H$ is not only $J$-adapted on $\mathscr{L}^{L_0}(M)$, but we may also take the $J$-holomorphic polygon $u$ appearing in the definition of a $J$-adapted metric to be a strip. Indeed, this appears as Corollaries 3.9 in \cite{BarraudCornea2006}. \par

The same is true of the spectral norm $\gamma$ and its variant $\gamma_\mathrm{ext}$ on $\mathscr{L}^{L_0}(M)$. This follows from the proof of Theorem E in \cite{KislevShelukhin2018}. When $M=T^*L_0$, and $\gamma$ is extended to $\mathscr{L}^{e}(T^*L_0)\cap \mathscr{L}^{m(1,0)}(T^*L_0)$, the same stay true. Indeed, the spectral norm between $L$ and $L'$ is then taken to be the usual one, but using different decorations on $L$ and $L'$. Therefore, the underlying manifolds are the same, and there is still a strip between them. \par

We remark that the polygons which are found in the proof of Theorem 5.0.2 of \cite{BiranCorneaShelukhin2018} and Theorem E in \cite{KislevShelukhin2018} are $J$-holomorphic for a very specific $J$. This choice is made to allow for the use Lelong's inequality. However, the same argument works perfectly well for any $J$ in an appropriate residual set. \par

\begin{rem} \label{rem:hbar}
When proving Theorem \ref{thm:main}, it is actually possible to remove the weak-exactness hypothesis on $\mathscr{L}^{L_0}(M)$ in the case of the spectral norm. Then $L_0$ can be monotone with any constant $\rho>0$ and $d\in\Z_2$. Indeed, the weak-exactness is only truly needed in order to prove the stronger Theorem \ref{thm:main-precise}. Denote by $\hbar(M,L,J)$ the minimal area of a $J$-holomorphic disk with boundary in $L$, or of a $J$-holomorphic sphere. Then, the strip found in the proof of Theorem E of \cite{KislevShelukhin2018} ---~and in Corollary 3.13 of \cite{BarraudCornea2006} for that matter~--- exists whenever $\gamma(L,L')<\hbar(M,L,J)$. As we will see below, this is enough to prove Theorem \ref{thm:main}. The same is true of the variant $\gamma_\mathrm{ext}$.
\end{rem}

\subsection{A version of the monotonicity lemma} \label{subsec:monotonicity}
We now prove a slightly improved version of Sikorav's monotonicity lemma for curves with Lagrangian boundary \cite{Sikorav1994}. The constants appearing in it will only depend on bounds coming from the Riemannian metric $g_J$ on $M$ and its restriction on Lagrangian submanifolds. This improvement follows from observations in Groman and Solomon \cite{GromanSolomon2016} on the dependence of the constant appearing in the isoperimetric inequality on these bounds. \par

Firstly, consider a loop $\gamma:\R/(2\pi\Z)\to B_{r_\mathrm{inj}(M)/2}(x)$, where $B_{r_\mathrm{inj}(M)/2}(x)$ is the metric ball in $M$ of radius $\frac{1}{2}r_\mathrm{inj}(M)$ centered at some $x\in M$. Denote by $a(\gamma)$ the area, with respect to $\omega$, of a disk extension of $\gamma$ contained in a metric ball $B_{r_\mathrm{inj}(M)/2}(y)$. Here, $y$ is not necessarily equal to $x$. This is of course independent on the choice of disk: if $u:\D\to B_{r_\mathrm{inj}(M)/2}(y)$ and $u':\D\to B_{r_\mathrm{inj}(M)/2}(y')$ are two such extensions, then gluing them along their common boundary gives a sphere $u\#\overline{u}$ in the ball $B_{r_\mathrm{inj}(M)}(\gamma(0))$. Such a sphere must be nullhomotopic. Therefore, $\omega(u)-\omega(u')=\omega(u\#\overline{u})=0$, and $a(\gamma)$ is well-defined. \par

Let $L\in\mathscr{L}^\star_{\Lambda,\epsilon}(M)$. Consider now an arc $\gamma:([0,\pi],\{0,\pi\})\to (M,L)$ with image in a metric ball $B_\delta(x)$ for some $x\in L$, where
\begin{align} \label{eqn:delta}
\delta:=\epsilon\min\left\{1,\frac{1}{2}r_\mathrm{inj}(M),\frac{1}{2}r_\mathrm{inj}(L)\right\}.
\end{align}
Take a path $\alpha:[0,\pi]\to B^L_{\delta/\epsilon}(y)$ such that $\gamma(\pi i)=\alpha(\pi i)$ for each $i\in\{0,1\}$. Here, $B^L_{\delta/\epsilon}(y)$ denotes the metric ball in $L$, i.e.\ with respect to $d_L$, of radius $\frac{\delta}{\epsilon}$ centered at $y\in L$. We then define $a(\gamma)$ to be the area, again with respect to $\omega$, of a disk extension of $\gamma\#\overline{\alpha}$ in the ball $B_\delta(y)$, i.e.\ $a(\gamma)=a(\gamma\#\overline{\alpha})$ as a loop. Here, $\overline{\alpha}(\theta):=\alpha(\pi-\theta)$ for all $\theta\in[0,\pi]$. \par

Note that such a path $\alpha$ always exists, since $d_M(x,\gamma(\pi i))<\epsilon$. Thus, it must be so that $d_L(x,\gamma(\pi i))<\min\{1,\frac{\delta}{\epsilon}\}$ if the tameness condition is to be fulfilled. We now show that $a(\gamma)$ is well-defined. Take $\alpha:[0,\pi]\to B^L_{\delta/\epsilon}(y)$ and $\alpha':[0,\pi]\to B^L_{\delta/\epsilon}(y')$ to be two paths such that $\alpha(\pi i)=\alpha'(\pi i)=\gamma(\pi i)$. Take $u:\D\to B_\delta(y)$ and $u':\D\to B_\delta(y')$ to be extensions of $\gamma\#\overline{\alpha}$ and $\gamma\#\overline{\alpha'}$, respectively. Then, gluing $u$ and $u'$ along $\gamma$ gives a disk $u\#\overline{u'}$ with boundary $\overline{\alpha}\#\alpha'$. But $\overline{\alpha}\#\alpha'$ is contained in $B^L_{2\delta/\epsilon}(\gamma(0))$. Since $\frac{2\delta}{\epsilon}\leq r_\mathrm{inj}(L)$, it must be a contractible loop. The homotopy from $\overline{\alpha}\#\alpha'$ to a point extends to a homotopy of $u\#\overline{u'}$ to a topological sphere in $B_{2\delta/\epsilon}(\gamma(0))$. Since $\frac{2\delta}{\epsilon}\leq r_\mathrm{inj}(M)$, this sphere must be nullhomotopic. Therefore, $\omega(u)-\omega(u')=\omega(u\#\overline{u'})=0$, and the definition of $a(\gamma)$ is again independent of choices. \par

\begin{lem}[Isoperimetric inequality] \label{lem:iso-inequality}
Let $M$, $L$ and $\delta$ be as above. There exist constants $c=c(K_0,r_0,\epsilon)>0$ and $c'=c'(K_0,r_0)>0$ such that
\begin{enumerate}[label=(\roman*)]
\item for all loops $\gamma$ with image in a metric ball $B_{r_0/2}(x)$ for some $x\in M$, we have that
\begin{align*}
a(\gamma)\leq c'\ell(\gamma)^2;
\end{align*}
\item for all arcs $\gamma$ with ends in $L$ and image in a metric ball $B_{\delta}(x)$ for some $x\in L$, we have that
\begin{align*}
a(\gamma)\leq c\ell(\gamma)^2.
\end{align*}
\end{enumerate}
Furthermore, $c'(K_0,r_0)\leq c(K_0,r_0,\epsilon)$.
\end{lem}
\begin{pr*}
As noted by Groman and Solomon \cite{GromanSolomon2016}, the proof appearing in Remark 4.4.2 of \cite{McDuffSalamon2012} depends only on the constants above. This gives a proof of \textit{(i)}. We give here the details of the proof of \textit{(ii)}. \par

Take $\alpha$ to be the unique minimizing geodesic in $L$ from $\gamma(0)$ to $\gamma(\pi)$; it exists since $d_L(\gamma(0),\gamma(\pi))<\frac{2\delta}{\epsilon}\leq r_\mathrm{inj}(L)$. We define an extension $u:\D\to M$ of $\gamma\#\overline{\alpha}$ by
\begin{align*}
u(re^{i\theta}):=\exp_{\gamma(0)}(r\xi(\theta)),
\end{align*}
where $\exp_{\gamma(0)}(\xi(\theta)):=(\gamma\#\overline{\alpha})(\theta)$ for all $\theta\in[0,2\pi]$. Although this disk might not be entirely contained in $B_{\delta}(x)$ ---~if $\delta$ is larger than the convexity radius of $M$ that is~--- it is contained in $B_{2\delta}(\gamma(0))$. Therefore, the same argument as above implies that $\omega(u)=a(\gamma)$. \par

By the Gauss lemma, we have that
\begin{align*}
|\del_r u|(re^{i\theta})=|(d\exp_{\gamma(0)})_{r\xi(\theta)}(\xi(\theta))|=|\xi(\theta)|=d_M(\gamma(0),(\gamma\#\overline{\alpha})(\theta)).
\end{align*}
However, for all $\theta\in[0,\pi]$, we have that
\begin{align*}
d_M(\gamma(0),(\gamma\#\overline{\alpha})(\theta))
&=d_M(\gamma(0),\gamma(\theta))\leq \ell(\gamma)
\intertext{and}
d_M(\gamma(0),(\gamma\#\overline{\alpha})(2\pi-\theta))
&=d_M(\alpha(0),\alpha(\theta))\leq \ell(\alpha) \\
&= d_L(\gamma(0),\gamma(\pi)) \leq \frac{1}{\epsilon}d_M(\gamma(0),\gamma(\pi)) \\
&\leq \frac{1}{\epsilon}\ell(\gamma).
\end{align*}
Therefore, we get that $|\del_r u|\leq \frac{1}{\epsilon}\ell(\gamma)$. \par

Likewise, by the Rauch comparison theorem, we have that
\begin{align*}
|\del_\theta u|(re^{i\theta})
&= |(d\exp_{\gamma(0)})_{r\xi(\theta)}(r\dot{\xi}(\theta))| \\
&\leq \frac{\sinh(r|\xi(\theta)|\sqrt{K_0})}{|\xi(\theta)|\sqrt{K_0}}|\dot{\xi}(\theta)| \\
&= \frac{\sinh(r|\xi(\theta)|\sqrt{K_0})}{|\xi(\theta)|\sqrt{K_0}}\left|(d\exp_{\gamma(0)})^{-1}_{\xi(\theta)}\left(\frac{d}{d\theta}(\gamma\#\overline{\alpha})(\theta)\right)\right| \\
&\leq \frac{\sinh(r|\xi(\theta)|\sqrt{K_0})}{\sin(|\xi(\theta)|\sqrt{K_0})}\left|\frac{d}{d\theta}(\gamma\#\overline{\alpha})(\theta)\right| \\
&\leq \frac{\sinh(r_0\sqrt{K_0})}{\sin(r_0\sqrt{K_0})}\left|\frac{d}{d\theta}(\gamma\#\overline{\alpha})(\theta)\right|.
\end{align*}

Therefore, we get that
\begin{align*}
a(\gamma)
&= \left|\int_0^{2\pi}\int_0^1\omega(\del_r u,\del_\theta u)rdrd\theta\right| \\
&\leq \int_0^{2\pi}\int_0^1|\del_r u||\del_\theta u|rdrd\theta \\
&\leq \frac{\sinh(r_0\sqrt{K_0})}{\epsilon\sin(r_0\sqrt{K_0})}\ell(\gamma)\int_0^{2\pi}\int_0^1 \left|\frac{d}{d\theta}(\gamma\#\overline{\alpha})(\theta)\right|rdrd\theta \\
&= \frac{\sinh(r_0\sqrt{K_0})}{2\epsilon\sin(r_0\sqrt{K_0})}\ell(\gamma)\left(\ell(\gamma)+\ell(\alpha)\right) \\
&\leq \left(1+\frac{1}{\epsilon}\right)\frac{\sinh(r_0\sqrt{K_0})}{2\epsilon\sin(r_0\sqrt{K_0})}\ell(\gamma)^2.
\end{align*}

This concludes to proof of \textit{(ii)}. To see that $c'\leq c$, one can notice that replacing $\ell(\alpha)$ by 0 in the above argument gives a proof of \textit{(i)} with $c'=\frac{\sinh(r_0\sqrt{K_0})}{2\sin(r_0\sqrt{K_0})}\leq c$.
\end{pr*}

Note that $\delta$ (as defined in (\ref{eqn:delta})) can always be bounded away from zero by a constant depending only on $K_0$, $r_0$, and $\Lambda$, as shows the following lemma. \par

\begin{lem}[\cite{GromanSolomon2016}] \label{lem:inj-rad}
Let $(M,g)$ be a complete Riemannian manifold with sectional curvature $K$ such that $|K|\leq K_0$, and with injectivity radius such that $r_\mathrm{inj}(M)\geq r_0>0$. Let $L$ be a submanifold with second fundamental form $B_L$ such that $||B_L||\leq \Lambda$ for some $\Lambda\geq 0$. Then, there exists a constant $i_0=i_0(K_0,r_0,\Lambda)>0$ such that
\begin{align*}
r_\mathrm{inj}(L,g|_L)\geq i_0.
\end{align*}
\end{lem}

\begin{rem} \label{rem:continuous}
The constant $c$ ---~and $c'$~--- we get in the proof depends continuously on $K_0$, $r_0$ and $\epsilon$. Likewise, the constant $i_0$ appearing in Lemma \ref{lem:inj-rad} may also be chosen so that it depends continuously on $K_0$, $r_0$ and $\Lambda$. This follows directly from the proof in Groman and Solomon's paper \cite{GromanSolomon2016}.
\end{rem}

\begin{prop}[Monotonicity lemma] \label{prop:monotonicity}
Let $M$, $L$, $\delta$ be as above. Let $\Sigma$ be a compact Riemann surface with boundary $\del\Sigma$ with corners. Consider a nonconstant $J$-holomorphic curve $u:(\Sigma,\del\Sigma)\to (B(x,r),\del B(x,r)\cup L)$ for some $x\in L$ and $r\leq \delta$ such that $x\in u(\Sigma)$. Suppose that $u$ sends the corners of $\Sigma$ to $\del B(x,r)\cap L$. Then, 
\begin{align*}
\omega(u)\geq Cr^2,
\end{align*}
where $C=\frac{1}{4c}$, and $c$ is the constant of Lemma \ref{lem:iso-inequality}.
\end{prop}
\begin{pr*}
The proof is that of Proposition 4.7.2 in \cite{Sikorav1994}, but using the version of the isoperimetric inequality above. We still give the details here for the sake of completeness. \par

Set $\Sigma_t:=u^{-1}(B(x,t))$ and $a(t):=\omega(u|_{\Sigma_t})$. By Sard's theorem, there is a subset of full measure $\Omega$ of $(0,r)$ such that for all $t\in\Omega$, $\Sigma_t$ is a subsurface of $\Sigma$ with piecewise smooth boundary $\del\Sigma_t=u^{-1}(\del B(x,t)\cup L)$. The discontinuities of the boundary are then contained in $u^{-1}(\del B(x,t)\cap L)$. \par

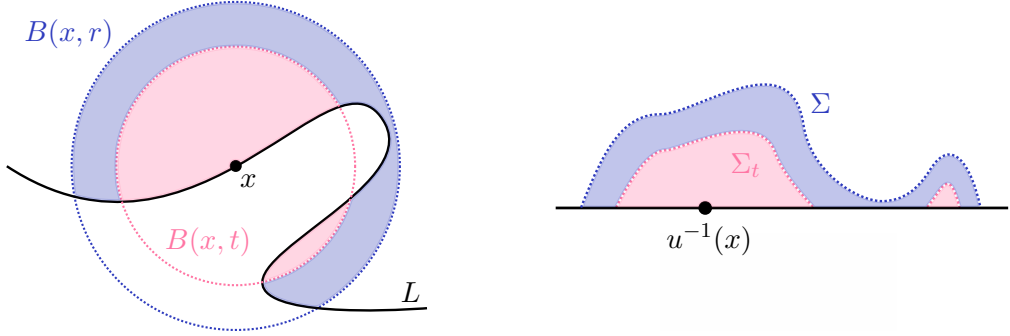
\begin{figure}[ht]
\centering
\begin{tikzpicture}
    \node[anchor=south west,inner sep=0] (image) at (0,0) {\includegraphics[width=0.4\textwidth]{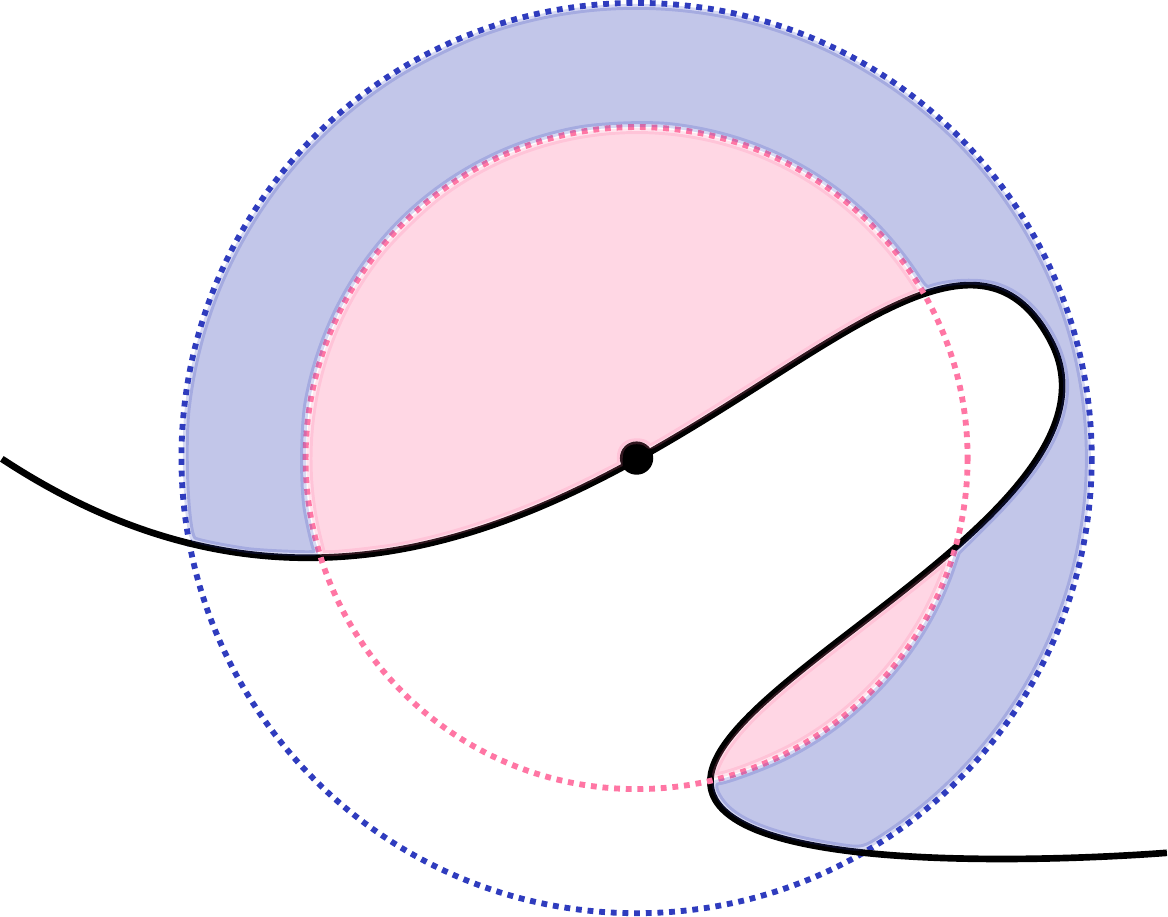}};
    \begin{scope}[x={(image.south east)},y={(image.north west)}]
        \node[] at (0.575,0.455) {$x$};
        \node[] at (0.96,0.12) {$L$};
        \node[FrenchPink] at (0.48,0.27) {$B(x,t)$};
        \node[DenimBlue] at (0.15,0.9) {$B(x,r)$};
    \end{scope}
\end{tikzpicture}
\qquad\qquad
\begin{tikzpicture}
    \node[anchor=south west,inner sep=0] (image) at (0,0) {\includegraphics[width=0.43\textwidth]{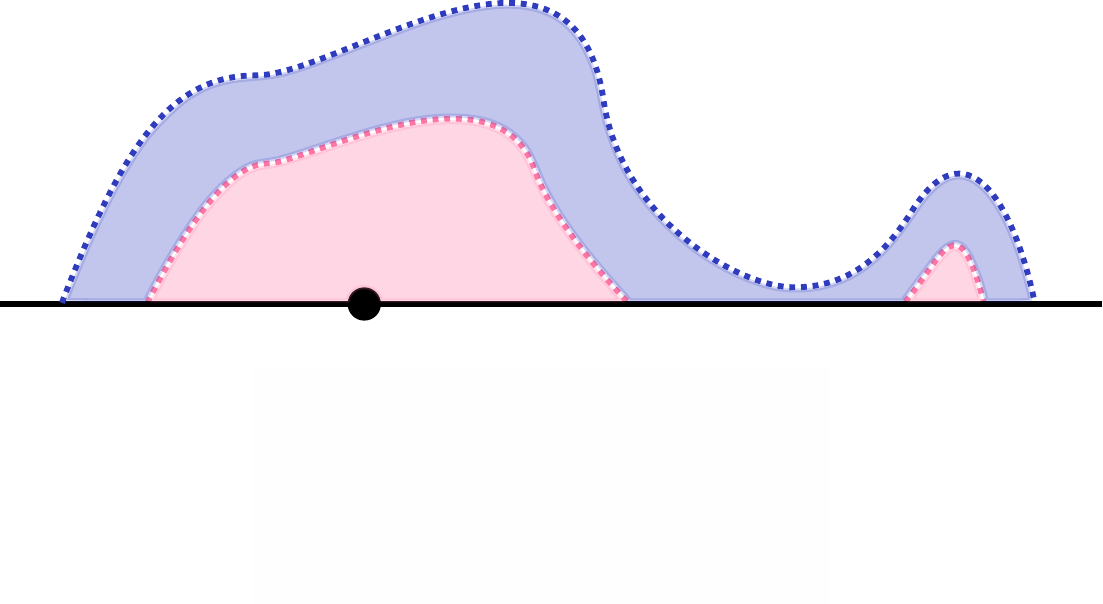}};
    \begin{scope}[x={(image.south east)},y={(image.north west)}]
        \node[] at (0.34,0.37) {$u^{-1}(x)$};
        \node[FrenchPink] at (0.42,0.67) {$\Sigma_t$};
        \node[DenimBlue] at (0.585,0.92) {$\Sigma$};
    \end{scope}
\end{tikzpicture}
\vspace*{8pt}
\caption{Visualization in two dimensions, both in the codomain (left) and domain (right).\label{fig:visualization-2d}}
\end{figure}

We begin by noting that, for $t\in\Omega$, we have the inequality
\begin{align*}
a(t)\leq c\ell(t)^2,
\end{align*}
where $\ell(t)$ is the length of $u|_{\del\Sigma_t-u^{-1}(L)}$. Indeed, write the boundary of $\Sigma_t$ as
\begin{align*}
\del\Sigma_t=\left(\bigsqcup_i \gamma_i\#\overline{\alpha_i}\right)\sqcup\left(\bigsqcup_j\beta_j\right),
\end{align*}
where the $\gamma_i$'s are arcs in the interior of $\Sigma$ with extremities in $u^{-1}(L)$, the $\alpha_i$'s are the segment of $u^{-1}(L)$ between the extremities of $\gamma_i$, and the $\beta_j$'s are loops. Finally, choose disk extensions $v_i:\D\to B_t(x)$ and $w_j:\D\to B_t(x)$ of $u|_{\gamma_i\#\overline{\alpha_i}}$ and $u|_{\beta_j}$, respectively. \par

Since $B_t(x)$ is contractible, we may take a primitive $\lambda$ of $\omega$ on it. Then, by Stokes' theorem,
\begin{align*}
a(t)
&= \int_{\Sigma_t} u^*\omega=\int_{\del \Sigma_t} u^*\lambda \\
&= \sum_i \int_{\gamma_i\#\overline{\alpha_i}} u^*\lambda+\sum_j \int_{\beta_j} u^*\lambda \\
&= \sum_i \int_{\del\D} v_i^*\lambda+\sum_j \int_{\del\D} w_j^*\lambda \\
&= \sum_i \int_{\D} v_i^*\omega+\sum_j \int_{\D} w_j^*\omega. \\
\intertext{Hence, by Lemma \ref{lem:iso-inequality},}
a(t)
&= \sum_i a(\gamma_i)+\sum_j a(\beta_j) \\
&\leq \sum_i c\ell(\gamma_i)^2+\sum_j c'\ell(\beta_j)^2 \\
&\leq c\ell(t)^2.
\end{align*}

Take $f:=\rho\circ u$, where $\rho$ is the distance from $x$ in $M$. Fix $t\in\Omega$. Then, in a neighborhood of $\del\Sigma_t$, $u^*g_J$ is a metric, and $f$ is its distance function from $u^{-1}(x)$. In particular, $|\grad f|_{u^*g_J}\equiv 1$. Therefore, by the coarea formula (see for example \cite{Chavel2006}), for $s$ near enough $t$, we have that
\begin{align} \label{eqn:coarea}
a(t)-a(s)
&= \int_{\{s\leq f\leq t\}}|\grad f|_{u^*g_J}da_{u^*g_J} \nonumber \\
&= \int_s^t\int_{\{f=\tau\}}d\ell_{u^*g_J}d\tau \\
&= \int_s^t \ell(\tau)d\tau. \nonumber
\end{align}
In particular, $a$ is differentiable on $\Omega$, and $a'(t)=\ell(t)$ for all $t\in\Omega$. \par

Therefore, for all such $t$, we have that
\begin{align*}
\left(\sqrt{a(t)}\right)'=\frac{a'(t)}{2\sqrt{a(t)}}=\frac{\ell(t)}{2\sqrt{a(t)}}\geq \frac{1}{2\sqrt{c}}
\end{align*}
by Lemma \ref{lem:iso-inequality}. Since $\Omega$ has full measure, we can integrate to get $\omega(u)=a(r)\geq \frac{1}{4c}r^2$.
\end{pr*}

\begin{rem} \label{rem:tame}
Both Lemma \ref{lem:iso-inequality} and Proposition \ref{prop:monotonicity} work for a larger class of almost complex structures $J$ and of metrics $g$ ---~not necessarily equal to $g_J$~--- respecting the hypotheses of Lemma \ref{lem:inj-rad}. Indeed, suppose that there exist constants $C_1,C_2>0$ such that $\omega(X,Y)\leq C_1 |X||Y|$ and $|X|^2\leq C_2\omega(X,JX)$ for all $X$ and $Y$. In that case, let $a$ denote the area with respect to $\omega$ as before, and let $A$ denote the area with respect to $g$. Then, Lemma \ref{lem:iso-inequality} gives the inequality $a(\gamma)\leq C_1c'\ell(\gamma)$ at \textit{(i)} and the inequality $a(\gamma)\leq C_1c\ell(\gamma)$ at \textit{(ii)}. Furthermore, Equation (\ref{eqn:coarea}) now applies to $A$, and thus $A'(t)=\ell(t)$. Therefore, Proposition \ref{prop:monotonicity} becomes
\begin{align*}
\omega(u)\geq \frac{1}{C_2}A(r)\geq \frac{1}{4C_1C_2c}r^2.
\end{align*}
\end{rem}

\subsection{Proof of Theorem \ref{thm:main-precise}} \label{subsec:proof}
We are now ready to give the proof of the main result. Remember that we have fixed a symplectic manifold $(M,\omega)$ and an $\omega$-compatible almost complex structure $J$ such that $(M,J)$ is either closed or convex at infinity. Likewise, we have fixed families of closed connected Lagrangian submanifolds $\mathscr{F}$ and $\mathscr{F}'$ such that $(\overline{\cup_{F\in\mathscr{F}}F})\cap (\overline{\cup_{F'\in\mathscr{F'}}F'})$ is discrete. We also assume from now on that all Lagrangian submanifolds are in some fixed collection $\mathscr{L}^\star_{\Lambda,\epsilon}(M)$ (as defined in Subsection \ref{subsec:notation}). In what follows, $\hat{d}^{\mathscr{F},\mathscr{F}'}$ denotes a $J$-adapted metric (as defined in Subsection \ref{subsec:j-adapted}) on that collection. \par

For subsets $A,B\subseteq M$, we take
\begin{align*}
s(A,B):=\sup_{x\in A} d_M(x,B):=\sup_{x\in A}\inf_{y\in B} d_M(x,y).
\end{align*}
Therefore, the Hausdorff metric is given by $\delta_H(A,B)=\max\{s(A,B),s(B,A)\}$. \par

Take $L,L'\in \mathscr{L}^\star_{\Lambda,\epsilon}(M)$ and $\delta>0$. Let $F_1,\dots,F_k$ be the Lagrangian submanifolds in $\mathscr{F}$ given by the definition of a $J$-adapted pseudometric. Take also $C^0$- and Hofer-small Hamiltonian perturbations $\widetilde{L}$, $\widetilde{L}'$ and $\widetilde{F}_1,\dots,\widetilde{F}_k$ of these manifolds making them pairwise transverse. Then, for any $x\in \widetilde{L}\cup\widetilde{L}'$, there exists a nonconstant $J$-holomorphic polygon $u:S_r\to M$ with the following properties:
\begin{enumerate}[label=(\roman*)]
\item it has boundary along $\widetilde{L}$, $\widetilde{L}'$, and $\widetilde{F}_1,\dots,\widetilde{F}_k$;
\item it passes through $x$;
\item $\omega(u)\leq d^\mathscr{F}(L,L')+\delta$.
\end{enumerate}
Assuming that the perturbations of $L$ and $L'$ are also $C^2$-small, we get that $\widetilde{L},\widetilde{L}'\in \mathscr{L}^\star_{\tilde{\Lambda},\tilde{\epsilon}}(M)$ for some $\tilde{\Lambda}\geq\Lambda$ and $\tilde{\epsilon}\leq\epsilon$. \par

Take $x\in\widetilde{L}-(\widetilde{L}'\cup \widetilde{F}_1\cup\dots \cup \widetilde{F}_k)$ and
\begin{align*}
\tilde{\delta}:=\epsilon\min\left\{1,\frac{1}{2}r_0,\frac{1}{2}i_0(K_0,r_0,\tilde{\Lambda})\right\},
\end{align*} 
where $i_0$ is the constant appearing in Lemma \ref{lem:inj-rad}. By Sard's theorem, there is an open dense subset of $(0,\min\{\tilde{\delta},d_M(x,\widetilde{L}'\cup \widetilde{F}_1\dots \cup \widetilde{F}_{|r|-2})\})$ such that, for all $\rho$ in this subset, $\Sigma:=u^{-1}(B_\rho(x))$ is a smooth submanifold of $S_r$ with boundary with corners. In particular, $u|_\Sigma$ respects the hypotheses of Proposition \ref{prop:monotonicity}. We thus get
\begin{align*}
d^\mathscr{F}(L,L')+\delta\geq \omega(u)\geq \omega(u|_\Sigma)\geq \tilde{C}\rho^2,
\end{align*}
where $\tilde{C}=C(K_0,r_0,\tilde{\Lambda})$ is the constant appearing in Proposition \ref{prop:monotonicity}. Since this holds for all $\delta>0$ and for all $\rho$ in a dense subset, we get
\begin{align*}
d^\mathscr{F}(L,L')\geq \tilde{C}\min\left\{\tilde{\delta},d_M(x,\widetilde{L}'\cup \widetilde{F}_1\dots \cup \widetilde{F}_{|r|-2})\right\}^2.
\end{align*}
In particular, whenever $d^\mathscr{F}(L,L')<\tilde{C}\tilde{\delta}^2$, we get
\begin{align} \label{eqn:dF-to-s_perturbations}
d^\mathscr{F}(L,L')\geq \tilde{C}s(\widetilde{L},\widetilde{L}'\cup \widetilde{F}_1\dots \cup \widetilde{F}_{|r|-2})^2
\end{align}
by taking the supremum over all $x$'s. \par

We now must get rid of the Hamiltonian perturbations on the right-hand side of (\ref{eqn:dF-to-s_perturbations}). In order to do so, choose sequences of generic $C^2$- and Hofer-small Hamiltonian diffeomorphisms $\{\phi_n\}_{n\geq 1}$ and $\{\phi^0_n\}_{n\geq 1}$, which converges to the identity in the $C^2$ sense. Likewise, for each $i\in\{1,\dots,k\}$, choose a sequence of generic $C^0$- and Hofer-small Hamiltonian diffeomorphisms $\{\phi^i_n\}_{n\geq 1}$ which converges to the identity in the $C^0$ sense. For each $n$, there is a $J$-holomorphic polygon $u_n:S_{r_n}\to M$ as above. \par

By Remark \ref{rem:continuous}, as $n$ tends to infinity, the corresponding constants $\tilde{\delta}$ and $\tilde{C}$ converge to
\begin{align*}
\delta_0:=\epsilon\min\left\{1,\frac{1}{2}r_0,\frac{1}{2}i_0(K_0,r_0,\Lambda)\right\}
\end{align*}
and $C=C(K_0,r_0,\Lambda)$, respectively. Furthermore, for any $n\geq 1$, $x\in L$, and $y\in L'\cup F_1\dots \cup F_k$, we have that
\begin{align*}
d_M(x,y)
&\leq d_M(x,\phi_n(x))+d_M(\phi_n(x),\phi^i_n(y))+d_M(\phi^i_n(y),y) \\
&\leq d_M(\phi_n(x),\phi^i_n(y)) + d_{C^0}(\Id,\phi_n) + \max_i d_{C^0}(\Id,\phi^i_n),
\end{align*}
for all $i\in\{0,1,\dots,k\}$. Therefore, by taking the infimum over all $y$ and then, the supremum over all $x$, we have that
\begin{align*}
s(\phi_n(L),\phi^0_n(L')\cup\phi^1_n(F_1)\cup\dots\cup\phi^k_n(F_k))
&\geq s(L,L'\cup F_1\cup\dots\cup F_k) \\
&\qquad - d_{C^0}(\Id,\phi_n) \\
&\qquad - \max_i d_{C^0}(\Id,\phi^i_n).
\end{align*}
We can thus finally put this back into (\ref{eqn:dF-to-s_perturbations}), and take the limit $n\to\infty$ to get
\begin{align} \label{eqn:dF-to-s}
d^\mathscr{F}(L,L')\geq C s(L,L\cup F_1\dots \cup F_k)^2 \geq C s(L,L'\cup (\cup_{F\in\mathscr{F}}F))^2,
\end{align}
whenever $d^\mathscr{F}(L,L')<C\delta_0^2$. In other words, Inequality (\ref{eqn:dF-to-s_perturbations}) holds without any perturbation. One gets similarly the inequalities
\begin{align} \label{eqn:dF'-to-s'}
d^\mathscr{F}(L,L') &\geq C s(L',L\cup (\cup_{F\in\mathscr{F}}F))^2 \nonumber \\
d^\mathscr{F'}(L,L') &\geq C s(L,L'\cup (\cup_{F'\in\mathscr{F'}}F'))^2 \\
d^\mathscr{F'}(L,L') &\geq C s(L',L\cup (\cup_{F'\in\mathscr{F'}}F'))^2. \nonumber
\end{align}
We must now turn the inequalities in (\ref{eqn:dF-to-s}) and (\ref{eqn:dF'-to-s'}) into an inequality in terms of $\delta_H(L,L')$. \par

Since $(\overline{\cup F})\cap (\overline{\cup F'})$ is discrete by the definition of a $J$-adapted metric, there exists a constant $\eta'>0$ such that $B_{\eta'}(x)\cap B_{\eta'}(y)=\emptyset$ for all $x\neq y\in (\overline{\cup F})\cap (\overline{\cup F'})$. Furthermore, there exists a constant $\eta''=\eta''(K_0,r_0,\Lambda)>0$ such that any closed manifold $L$ which is contained in a metric ball $B_{\eta''}(x)$, for some $x\in M$, must have $||B_L||>\Lambda$ ---~we refer the reader to Corollary \ref{cor:bound_2nd} in the next section for a precise estimate of $\eta''$. We set $\eta:=\min\{\eta',\eta''\}$. \par

\begin{figure}[ht]
\centering
\begin{tikzpicture}
    \node[anchor=south west,inner sep=0] (image) at (0,0) {\includegraphics[width=0.7\textwidth]{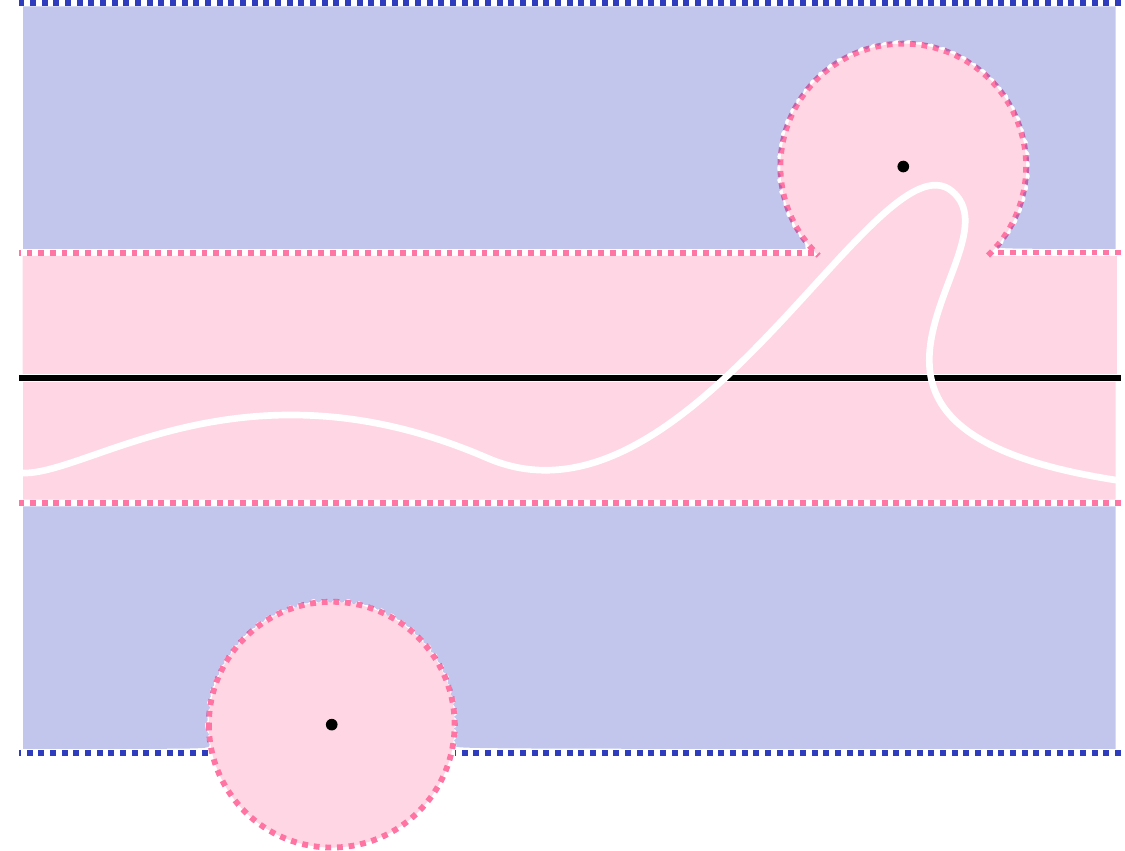}};
    \begin{scope}[x={(image.south east)},y={(image.north west)}]
        \node[white] at (0.48,0.49) {$L$};
        \node[] at (0.21,0.595) {$L'\cup((\cup F)\cap (\cup F'))$};
    \end{scope}
\end{tikzpicture}
\vspace*{8pt}
\caption{The $3\sigma$-neighborhood of $L'$ (in pink) contains the connected component containing $L'$ of the $\sigma$-neighborhood of $L'\cup ((\overline{\cup F})\cap (\overline{\cup F'}))$ (in blue).\label{fig:neighborhoods}}
\end{figure}

Suppose now that $\hat{d}^{\mathscr{F},\mathscr{F'}}(L,L')<\min\{C\delta_0^2,C\eta^2\}=:R$. By the inequalities in (\ref{eqn:dF-to-s}) and (\ref{eqn:dF'-to-s'}), we have that 
\begin{align*}
\sigma:=\max\{s(L,L'\cup (\cup F)),s(L,L'\cup (\cup F'))\}<\eta.
\end{align*}
But, by definition of $s$, $L$ must be in the $\sigma$-neighborhood of $L'\cup (\overline{\cup F})$ and $L'\cup (\overline{\cup F'})$, and thus of $L'\cup ((\overline{\cup F})\cap(\overline{\cup F'}))$. However, this neighborhood is composed of a disjoint union of an open neighborhood of $L'$ and metric balls of radius $<\eta''$. By the choice of $\eta''$, $L$ must then be in the component containing $L'$. However, this component is itself contained in the $3\sigma$-neighborhood of $L'$. Therefore, $3\sigma\geq s(L,L')$. By taking the maximum of the inequality in (\ref{eqn:dF-to-s_perturbations}) and the second inequality of (\ref{eqn:dF'-to-s'}), we get
\begin{align*}
\hat{d}^{\mathscr{F},\mathscr{F'}}(L,L')\geq C\sigma^2\geq \frac{C}{9}s(L,L')^2.
\end{align*}
Doing the same thing for the first and third inequalities of (\ref{eqn:dF'-to-s'}) finishes the proof. The $C$ appearing in the statement of Theorem \ref{thm:main-precise} is thus equal to $\frac{1}{36c}$, where $c$ is the constant of Lemma \ref{lem:iso-inequality}. \par

\begin{rem} \label{rem:convergence}
\begin{enumerate}[label=(\arabic*)]
\item If one is only interested by the statement of Theorem \ref{thm:main} and is willing to restrict oneself to the case when $M$ is compact, then the standing hypothesis that $(\overline{\cup F})\cap (\overline{\cup F'})$ is discrete can be loosened up. More precisely, we can suppose that $(\overline{\cup F})\cap (\overline{\cup F'})$ is only totally disconnected. \par

In that case, a sequence $\{L_n\}_{n\geq 1}$ must have a converging subsequence; this is because the collection of closed subsets of $M$ with the Hausdorff metric is compact. Denote by $E$ its limit set. Since the inequalities in (\ref{eqn:dF-to-s}) and (\ref{eqn:dF'-to-s'}) still stands ---~with $L$ replaced by $L_n$ and $L'$ by $L_0$~--- the limit $E$ must lie in $L_0\cup((\overline{\cup F})\cap (\overline{\cup F'}))$. However, a sequence of connected subsets must converge to a connected one (see \cite{Stacy1967} for instance). But, by the argument above, it cannot converge to a point. Therefore, we must have $E\subseteq L_0$. Using an analogous argument for the first and last inequality in (\ref{eqn:dF'-to-s'}), we get $L_0\subseteq E\cup((\overline{\cup F})\cap (\overline{\cup F'}))$. Thus, $L_0\subseteq E$ by the same connectivity argument. Therefore, any converging subsequence of $\{L_n\}_{n\geq 1}$ converges to $E=L_0$, which means that $\{L_n\}_{n\geq 1}$ itself converges to $L_0$ in the Hausdorff metric. \par

\item Likewise, to prove Theorem \ref{thm:main}, there is no need for the weak-exactness hypothesis on $\mathscr{L}^{L_0}(M)$ in the spectral case. Indeed, for $n$ large, we will have $\gamma(L_0,L_n)<\hbar(M,L_0,J)$. Therefore, we will still have a $J$-holomorphic strip between Hamiltonian deformations of $L_0$ and $L_n$. The rest of the proof then follows as above. The same is true of $\gamma_\mathrm{ext}$. \par
\end{enumerate}
\end{rem}

\section{The two-dimensional case} \label{sec:proof-dim2}
We now explain how to prove the second part of Theorem \ref{thm:main}. More precisely, we will prove the following.

\begin{thm} \label{thm:main-precise_dim2}
Suppose that $\dim M=2$. For any $J$-adapted metric $\hat{d}^{\mathscr{F},\mathscr{F'}}$ on $\mathscr{L}^\star_{\Lambda}(M)$, there exist constants $R'=R'(K_0,r_0,\Lambda,(\overline{\cup F})\cap (\overline{\cup F'}))>0$ and $C'=C'(K_0,r_0)>0$ such that whenever $\hat{d}^{\mathscr{F},\mathscr{F'}}(L,L')<R$, then
\begin{align*}
\hat{d}^{\mathscr{F},\mathscr{F'}}(L,L')\geq C'\delta_H(L,L')^2.
\end{align*}
\end{thm}

As noted in the introduction, the key to getting rid of the dependency of both $R$ and $C$ on $\epsilon$ is to use the two-dimensionality of $M$ in order to make a better choice of a metric ball in $M$ before applying the following version of the monotonicity lemma. \par

\begin{prop}[Monotonocity lemma, absolute version] \label{prop:monotonicity_absolute}
Let $(M,\omega,J,g_J)$ be a symplectic manifold equipped with a compatible almost complex structure and the corresponding metric. Let $\Sigma$ be a compact Riemann surface with boundary $\del\Sigma$. Consider a nonconstant $J$-holomorphic curve $u:(\Sigma,\del\Sigma)\to (B(x,r),\del B(x,r))$, where $x\in M$ and $r\leq \frac{1}{2}r_0$. Suppose that $x\in u(\Sigma)$. Then, 
\begin{align*}
\omega(u)\geq C'r^2,
\end{align*}
where $C'=\frac{1}{4c'}$, with $c'$ the constant of Lemma \ref{lem:iso-inequality}.
\end{prop}
\begin{pr*}
We make use of the same notation as in the proof of Proposition \ref{prop:monotonicity}. Since $u$ does not have boundary components along some Lagrangian submanifold $L$, the boundary $\del\Sigma_t$ is only composed of disjoint loops $\beta_j$. Therefore, we now get
\begin{align*}
a(t)= \sum_j a(\beta_j)\leq \sum_j c'\ell(\beta_j)^2=c'\ell(t)^2.
\end{align*}
The rest of the proof than follows as in Proposition \ref{prop:monotonicity}.
\end{pr*}

Let $L$ and $L'$ be Lagrangian submanifolds of $M$, i.e.\ closed curves on $M$, and let $\widetilde{L}$ and $\widetilde{L}'$ be sufficiently small generic Hamiltonian deformations of them. Take $x\in \widetilde{L}\cup \widetilde{L}'$. Let $u:S_r\to M$ be the $J$-holomorphic polygon with boundary along $L$ and $L'$ ---~and potentially some other Lagrangian submanifolds~--- passing through $x$ given by the definition of a $J$-adapted metric. By Proposition \ref{prop:monotonicity_absolute}, if we can find a disk entirely contained in $u(\operatorname{int}S_r)$, then we will get a proof of Theorem \ref{thm:main-precise_dim2}. Indeed, it will then suffice to change the step in the proof of Theorem \ref{thm:main-precise} where we apply the monotonicity lemma; the rest of the proof remains the same. It thus suffices to prove the following purely geometric result. \par

\begin{thm} \label{thm:disk_dim2}
Let $M$ be a complete surface with Gaussian curvature $|K|\leq K_0$ and injectivity radius $r_\mathrm{inj}(M)\geq r_0>0$. Let $L,K_1,\dots,K_k$ be pairwise transversal closed curves on $M$ with geodesic curvature uniformly bounded by $\Lambda$. Let $x\in L-\cup_i K_i$. Suppose that there exists a smooth map $u:S_r\to M$ with boundaries along $L,K_1,\dots,K_k$ passing through $x$. Then, there exists a constant
\begin{align*}
\rho_0=\rho_0(K_0,r_0,\Lambda,d_M(x,\cup_i K_i))>0
\end{align*}
such that $u(\operatorname{int}S_r)$ contains a metric ball of radius $\rho_0$.
\end{thm}

We allow $n$ to be zero, in which case it is understood that $S_r=\D$, $\cup_i K_i=\emptyset$, and $d_M(x,\cup_i K_i)=\infty$ for any $x\in L$. \par

Although it was developed independently, our proof of Theorem \ref{thm:disk_dim2} uses a similar approach to a recent proof by Petrunin and Zamora Barrera of the so-called Moon in a puddle theorem \cite{PetruninZamora2021}. We recover their result by taking $M$ to be the Euclidian plane and $n$ to be zero. Indeed, one can check that, in this case, $\rho_0=\frac{1}{\Lambda}$. \par

\begin{rem} \label{rem:rho_0-continuous}
Much like the constants $C$, $C'$ and $i_0$, the constant $\rho_0$ depends continuously on $K_0$, $r_0$, $\Lambda$, and $d_M(x,\cup_i K_i)$.
\end{rem}

\subsection{Finding a good disk}  \label{subsec:finding-disk}
The proof of Theorem \ref{thm:disk_dim2} relies mostly on the following technical lemma, whose proof we will delay until the next subsection. \par

\begin{lem} \label{lem:disk_curvature}
Let $M$ be a complete Riemannian manifold of dimension $n\geq 2$ with $|K|\leq K_0$ and $r_\mathrm{inj}(M)\geq r_0>0$. There exist constants $\rho_1=\rho_1(K_0,r_0)>0$ and $\alpha=\alpha(K_0,r_0)>0$ with the following property. Let $\Gamma:[0,\ell]\to M$ be a unit-speed curve with image in $B_\rho(x)$ for some $0<\rho\leq \rho_1$ and some $x\in M$. Consider the map
\begin{center}
\begin{tikzcd}[row sep=0pt,column sep=1pc]
 d\colon [0,\ell] \arrow{r} & \R_{\geq 0} \\
  {\hphantom{d\colon{}}} s \arrow[mapsto]{r} & d_M(x,\Gamma(s)).
\end{tikzcd}
\end{center}
Suppose that $d$ achieve its maximum at $s_0\in (0,\ell)$. Then,
\begin{align*}
\left|\frac{D}{ds}\dot{\Gamma}(s_0)\right|\geq \frac{\alpha}{\rho}.
\end{align*}
\end{lem}

By taking $M=\R^2$ and the limit $K_0\to 0$, we recover the classical fact that, on a loop contained in a disk of radius $\rho>0$, there is a point were its curvature is at least $\frac{1}{\rho}$. \par

Before going in the proof of Theorem \ref{thm:disk_dim2}, we note that Lemma \ref{lem:disk_curvature} gives a precise bound on the smallest metric ball that a submanifold $L$ with second fundamental form $||B_L||\leq\Lambda$ can be contained in. \par

\begin{cor} \label{cor:bound_2nd}
Let $M$, $K_0$, $\rho_0$ and $\alpha$ be as above. Let $L$ be a closed submanifold of $M$ contained in a metric ball $B_\rho(x)$ for some $0<\rho\leq \rho_0$ and some $x\in M$. Then, its second fundamental form respects
\begin{align*}
||B_L||\geq \frac{\alpha}{\rho}.
\end{align*}
\end{cor}
\begin{pr*}
Since $L$ is closed, there is a point $y\in L$ such that the map $x'\mapsto d_M(x,x')$, seen as a map $L\to\R$, achieve its maximum at $y$. Let $\gamma:(-\epsilon,\epsilon)\to L$ be a unit-speed geodesic of $L$ such that $\gamma(0)=y$. By Lemma \ref{lem:disk_curvature}, we have that
\begin{align*}
\left|B_L\left(\dot{\gamma}(0),\dot{\gamma}(0),\frac{\frac{D}{ds}\dot{\gamma}}{|\frac{D}{ds}\dot{\gamma}|}(0)\right)\right|=\left|\frac{D}{ds}\dot{\gamma}(0)\right|\geq \frac{\alpha}{\rho}.
\end{align*}
\end{pr*}

Before going in the proof of Theorem \ref{thm:disk_dim2}, we reduce to the case where $M$ is simply connected. Suppose that it is not. We then consider its universal cover $\pi:\widetilde{M}\to M$. Note that $\widetilde{M}$ comes naturally equipped with a Riemannian metric $\pi^*g$, which turns $\pi$ into a local isometry. Therefore, we have that the Gaussian curvature of $\widetilde{M}$ respects $|\tilde{K}|\leq K_0$ and that $r_\mathrm{inj}(\widetilde{M})\geq r_\mathrm{inj}(M)\geq r_0$. The inequality between the injectivity radii follows from the classical result of Klingenberg \cite{Klingenberg1959} that
\begin{align} \label{eqn:r_inj}
r_\mathrm{inj}(M)=\min\left\{r_\mathrm{conj}(M),\frac{1}{2}\ell(M)\right\}.
\end{align}
Here, $r_\mathrm{conj}(M)$ is the length of the shortest geodesic segment $\gamma:[0,T]\to M$ such that there exists a normal Jacobi field $J$ along $\gamma$ with $J(0)=J(T)=0$. Likewise, $\ell(M)$ is the length of the shortest \emph{immersed} geodesic loop on $M$. \par

Since $S_r$ is contractible, $u$ admits a lift $\tilde{u}:S_r\to\widetilde{M}$. Then, $d_M(x,\cup_i K_i)$ gives a lower bound on the distance between $\tilde{x}\in \tilde{u}(S_r)$ such that $\pi(\tilde{x})=x$ and the components of $\tilde{u}(\del S_r)$ that do not contain $\tilde{x}$. \par

Therefore, the metric bounds on $M$ are also respected by $\widetilde{M}$. It thus indeed suffices to prove Theorem \ref{thm:disk_dim2} on the universal cover: if $B_{\rho_0}(\tilde{x})$ is the metric ball in $\widetilde{M}$ given by the theorem, then $B_{\rho_0}(\pi(\tilde{x}))=\pi(B_{\rho_0}(\tilde{x}))$ will be the sought-after ball in $M$. Note that this is indeed a topological ball as we may take $\rho_0\leq r_0$. Therefore, for the rest of the proof, we will assume that $M$ is simply connected. \par

We now fix an injective unit-speed parametrization $\Gamma:[0,\ell]\to M$ of the component of $u(\del S_r)$ containing $x$, i.e.\ $\Gamma$ parametrize a segment of $L$ containing $x$, and $\Gamma(0),\Gamma(\ell)\in L\cap (\cup_{i=1}^n K_i)$ when $n>0$. When $n=0$, $\Gamma$ is just a parametrization of $L$. We then take the convention that $\Gamma(0)=\Gamma(\ell)\neq x$. \par

The proof has three main steps.
\begin{enumerate}[label=(\arabic*)]
\item We define a notion of an "osculating disk" $D_\rho(s)$ of the curve $\Gamma$ at $s$. This disk is a closed metric ball of $M$ which has the property that $D_\rho(s)\cap\Gamma([s-\epsilon,s+\epsilon])\subseteq \del D_\rho(s)$ for some $\epsilon>0$.
\item We find $s_0,t_0\in [0,\ell]$ with the following property: if $p\in \operatorname{int}(D_\rho(s))\cap\Gamma([0,\ell])$ for any $s\in [s_0,t_0]$, then there also exists $t\in [t_0,s_0]$ such that $\Gamma(t)\in \operatorname{int}(D_\rho(s))$.
\item We suppose that $\operatorname{int}(D_\rho(s))\cap\Gamma([0,\ell])\neq\emptyset$ for all $s\in [s_0,t_0]$, and we get a contradiction.
\end{enumerate}

Denote by $r_\mathrm{conv}(M)$ the convexity radius of $M$, i.e.\ the largest $\rho>0$ such that, for all $x\in M$ and all $y,y'\in B_\rho(x)$, there exists a minimizing geodesic from $y$ to $y'$ in $B_\rho(x)$. It is a classical result from Berger \cite{Berger1976} that
\begin{align*}
r_\mathrm{conv}(M)\geq \frac{1}{2}\min\left\{r_\mathrm{inj}(M),\frac{\pi}{\sqrt{K_0}}\right\}.
\end{align*}
By taking a smaller $\rho_0$ if necessary, we may assume that $\rho_0\leq \min\{\frac{r_0}{2},\frac{\pi}{2\sqrt{K_0}},\rho_1\}$. Therefore, $\rho\leq \rho_0$ implies that $\rho_0\leq r_\mathrm{conv}(M)$ and that $\rho$ respects Lemma \ref{lem:disk_curvature}. \par

\begin{rem} \label{rem:convexity-radius}
In fact, a more recent result of Dibble \cite{Dibble2017} gives
\begin{align}
r_\mathrm{conv}(M)=\min\left\{r_\mathrm{foc}(M),\frac{1}{4}\ell(M)\right\},
\end{align}
where $r_\mathrm{foc}(M)$ is the length of the shortest geodesic segment $\gamma:[0,T]\to M$ such that there exists a normal Jacobi field $J$ along $\gamma$ with $J(0)=\langle J',J \rangle(T)=0$. This can be used to give a better estimate on the optimal $\rho_0$ for a given $(M,g_J)$.
\end{rem}

We thus begin with the definition of $D_\rho(s)$. Let $s\in (0,\ell)$ and $\rho\in (0,\rho_0)$. We define $D_\rho(s)$ to be the closed metric ball $\overline{B_\rho(\gamma_s(\rho))}$. Here, $\gamma_s(t):=\exp_{\Gamma(s)}(tN(s))$ for $t\in [0,r_\mathrm{inj}(M)]$, and $N$ is the unit-length vector field along $\Gamma$ which is orthogonal to $\dot{\Gamma}$ and pointing toward the interior of the topological disk $\overline{u(S_r)}$. \par

\begin{lem} \label{lem:disk_osculating}
Let $s\in (0,\ell)$ and $\rho\in (0,\min\{\frac{\alpha}{\Lambda},\rho_1\})$, where $\rho_1$ and $\alpha$ are the constant appearing in Lemma \ref{lem:disk_curvature}. There exists $\epsilon>0$ such that
\begin{align*}
D_\rho(s)\cap \Gamma([s-\epsilon,s+\epsilon])\subseteq \del D_\rho(s).
\end{align*}
\end{lem}

We also put off the proof of the lemma until the next subsection. In order for Lemma \ref{lem:disk_osculating} to stand, we add the condition that $\rho_0\leq \frac{\alpha}{\Lambda}$ and take $\rho\in (0,\rho_0)$. \par

We now make our choice of $t_0$ and $s_0$. Let $t_0\in (0,\ell)$ be such that $\Gamma(t_0)=x$. If we have that $D_\rho(t_0)\cap \Gamma([0,\ell])\subseteq \del D_\rho(t_0)$, then we have proven the theorem. Indeed, for all $y\in \operatorname{int}(D_\rho(t_0))$,
\begin{align*}
d(y,u(\del S_r)-\Gamma([0,\ell]))\geq d(x,\cup_i K_i)-d(x,y)> \rho-\rho=0.
\end{align*}
Thus, $\operatorname{int}D_\rho(t_0)\cap u(\del S_r)=\emptyset$. The open ball $\operatorname{int}(D_\rho(t_0))$ must then be in $u(\operatorname{int} S_r)$ by construction, and it thus satisfies the conclusions of Theorem \ref{thm:disk_dim2}. \par

Suppose therefore that $D_\rho(t_0)\cap \Gamma([0,\ell])\nsubseteq \del D_\rho(t_0)$. Fix $s_0\in (0,\ell)$ such that
\begin{align*}
d_M(\Gamma(t_0),\Gamma(s_0))=\min\{d_M(\Gamma(t_0),\Gamma(s))\mid s\in [0,\ell], \Gamma(s)\in D_\rho(t_0)\}.
\end{align*}
We may indeed take $s_0$ to be in $(0,\ell)$, since $\Gamma(0),\Gamma(\ell)\notin D_\rho(t_0)$ by the hypotheses on $\rho$ when $n>0$. When $n=0$, $\Gamma$ is actually a loop, and we can still assume it by shifting the parametrization $\Gamma$. Indeed, if all of $L$ were in $D_\rho(t_0)$, then we would get a contradiction by Corollary \ref{cor:bound_2nd}. Furthermore, we must have that $\Gamma(s_0)\in\operatorname{int} D_\rho(t_0)$ and, by Lemma \ref{lem:disk_osculating}, $s_0\neq t_0$. Without loss of generality, we may assume that $s_0>t_0$. \par

Denote by $[\Gamma(t_0),\Gamma(s_0)]$ the unique minimizing geodesic segment in $D_\rho(t_0)$ from $\Gamma(t_0)$ to $\Gamma(s_0)$. By hypothesis on $\rho$, it exists and does not intersect $u(\del S_r)$ ---~except at its extremities of course. Then, $\Gamma([t_0,s_0])\cup [\Gamma(t_0),\Gamma(s_0)]$ is a continuously embedded loop in $M$. Therefore, since $M$ is simply connected, it bounds a topological disk and divides $M$ in two parts. \par

We now get to the proof that $[t_0,s_0]$ has the desired property, i.e.\ the proof of the second step. Suppose that there exists $p\in u(\del S_r)-\Gamma([t_0,s_0])$ such that $p\in D_\rho(s)$ for some $s\in [t_0,s_0]$ ---~obviously there would be nothing to prove if $p\in \Gamma([t_0,s_0])$. Since the topological disk $\Gamma([t_0,s_0])\cup [\Gamma(t_0),\Gamma(s_0)]$ divides $M$ in two, the minimizing geodesic from $\Gamma(s)$ to $p$ must intersect $\Gamma([t_0,s_0])\cup [\Gamma(t_0),\Gamma(s_0)]$ at some point $y$. This geodesic exists, is unique and is contained in $D_\rho(s)$, since $\rho<r_\mathrm{conv}(M)\leq r_\mathrm{inj}(M)$. Note that since $p\neq \Gamma(s)$ by hypothesis, we have that $y\in\operatorname{int}(D_\rho(s))$ and that $d_M(p,y)<d_M(p,\Gamma(s))$. \par

If $y\in \Gamma([t_0,s_0])$, then we have $y=\Gamma(t)\in D_r(s)$ for some $t\in [s_0,t_0]$ as desired. Suppose therefore that $y\in [\Gamma(t_0),\Gamma(s_0)]$. Clearly, it suffices to prove that either $\Gamma(t_0)$ or $\Gamma(s_0)$ is in $D_\rho(s)$. We thus suppose that it is not the case and get a contradiction. In this case, $[\Gamma(t_0),\Gamma(s_0)]$ divides $D_\rho(s)$ in exactly two parts, both with nonempty interior. Indeed, $y\in [\Gamma(t_0),\Gamma(s_0)]\cap\operatorname{int}(D_\rho(s))$. Moreover, if $D_\rho(s)-[\Gamma(t_0),\Gamma(s_0)]$ had more than two components, we would get a contradiction. To see this, note that it is equivalent to $[\Gamma(t_0),\Gamma(s_0)]\cap D_\rho(s)$ having more than one component. However, if $q, q'\in [\Gamma(t_0),\Gamma(s_0)]\cap D_\rho(s)$ were to belong to different components, there would still be a minimizing geodesic segment $[q,q']$ in $D_r(s)$. But then
\begin{align*}
\ell([\Gamma(t_0),q]\cup[q,q']\cup [q',\Gamma(s_0)])<\ell([\Gamma(t_0),\Gamma(s_0)])=d_M(\Gamma(t_0),\Gamma(s_0)),
\end{align*}
which is of course impossible. Furthermore, since
\begin{align*}
d_M(\Gamma(t_0),\Gamma(s_0))\leq d_M(\Gamma(t_0),\gamma_s(\rho))+d_M(\gamma_s(\rho),\Gamma(s_0))<2\rho,
\end{align*}
the segment $[\Gamma(t_0),\Gamma(s_0)]$ cannot intersect $\gamma_s(\rho)$. Therefore, $\gamma_s(\rho)$ must lie in the interior of one of the component of $D_\rho(s)-[\Gamma(t_0),\Gamma(s_0)]$. Furthermore, $\Gamma(s)$ must also lie in the interior of a component. Otherwise, we would have that $\Gamma(s)\in [\Gamma(t_0),\Gamma(s_0)]$, and then $\Gamma(s)$ would be a point of $D_\rho(t_0)$ with $d_M(\Gamma(t_0),\Gamma(s))<d_M(\Gamma(t_0),\Gamma(s_0))$, contradicting the very definition of $s_0$. \par

Suppose at first that $\gamma_s(\rho)$ lies in the same component as $\Gamma(s)$. We begin by showing that this implies that $p\in D_\rho(t_0)$. Suppose that it does not. Denote by $A$ the intersection of $\del D_\rho(s)$ and the component of $D_\rho(s)-[\Gamma(t_0),\Gamma(s_0)]$ not containing $\gamma_s(\rho)$. By minimality of $[\Gamma(t_0),\Gamma(s_0)]$, $A$ is an embedded arc. Furthermore, since $\del A\subseteq [\Gamma(t_0),\Gamma(s_0)]$, we have that $\del A\subseteq D_\rho(t_0)$. However, since $p\notin D_\rho(t_0)$, $A$ is not entirely contained in $D_\rho(t_0)$. In particular, the function $q\mapsto d_M(\gamma_{t_0}(\rho),q)$, $q\in A$, achieves its maximum at a point $q\in A-\del A$. This implies that $A$ must be normal to the geodesic segment $[\gamma_{t_0}(\rho),q]$ at $q$ ---~to see this, one can consider the energy functional along the minimal geodesics from $\gamma_{t_0}(\rho)$ to $A$ and use the fact that it has a critical point at $q$. But by construction, $A$ is also normal to the geodesic segment $[\gamma_s(\rho),q]$. Since $\dim M=2$ and $d(\gamma_{t_0}(\rho),q)>\rho=d(\gamma_s(\rho),q)$, this means that $[\gamma_s(\rho),q]\subsetneq [\gamma_{t_0}(\rho),q]$. \par

Denote by $q'$ the intersection of the geodesic segments $[\gamma_{t_0}(\rho),q]$ and $[\Gamma(t_0),\Gamma(s_0)]$. Suppose at first that $d_M(\Gamma(t_0),q')\leq\rho$. Then, $[\gamma_{t_0}(\rho),q']\subseteq \overline{B_\rho(\Gamma(t_0))}$, since $\rho<r_\mathrm{conv}(M)$. However, $\gamma_s(\rho)\in [\gamma_{t_0}(\rho),q']$, and thus $d_M(\Gamma(t_0),\gamma_s(\rho))\leq \rho$. This leads to a contradiction since we have supposed that $\Gamma(t_0)\notin D_\rho(s)$. Therefore, it must be so that $d_M(\Gamma(t_0),q')>\rho$. In particular,
\begin{align*}
d_M(\Gamma(s_0),q')= d_M(\Gamma(s_0),\Gamma(t_0))-d_M(q',\Gamma(t_0))<2\rho-\rho=\rho
\end{align*}
However, by definition of $s_0$, $d_M(\gamma_{t_0}(\rho),\Gamma(s_0))<\rho$. Therefore, we have that $[\gamma_{t_0}(\rho),q']\subseteq B_\rho(\Gamma(s_0))$ similarly as before. Thus, $d_M(\gamma_s(\rho),\Gamma(s_0))<\rho$ and $\Gamma(s_0)\in D_\rho(s)$, which is again a contradiction. Therefore, it must be so that $p\in D_\rho(t_0)$. \par

We now finally prove that $\gamma_s(\rho)$ cannot lie in the same component of $D_\rho(s)-[\Gamma(t_0),\Gamma(s_0)]$ as $\Gamma(s)$. Recall that $y$ is the intersection of the geodesic segments $[\Gamma(t_0),\Gamma(s_0)]$ and $[\gamma_s(\rho),p]$. Denote by $z$ the intersection of $[y,\Gamma(s_0)]$ and $\del D_\rho(s)$ ---~once again, minimality of the geodesic ensures that this is indeed a point. By definition of $y$, we have that
\begin{align*}
d_M(y,p)
&= d_M(\gamma_s(\rho),p)-d_M(\gamma_s(\rho),y) \\
&\leq \rho-d_M(\gamma_s(\rho),y) \\
&\leq d_M(\gamma_s(\rho),y)+d_M(y,z)-d_M(\gamma_s(\rho),y) \\
&< d_M(y,\Gamma(s_0)).
\end{align*}
Indeed, $[\gamma_s(\rho),y]\cup[y,z]$ is a path from the center of a ball of radius $\rho$ to its boundary, and thus must be of length at least $\rho$. \par

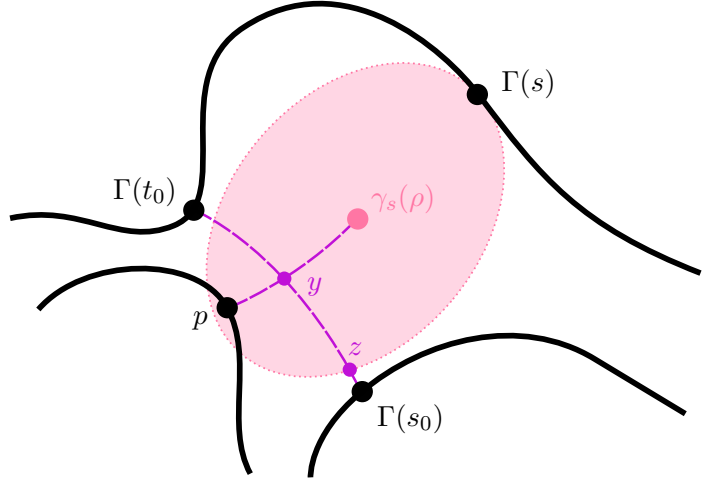
\begin{figure}
\centering
\begin{tikzpicture}
    \node[anchor=south west,inner sep=0] (image) at (0,0) {\includegraphics[width=0.7\textwidth]{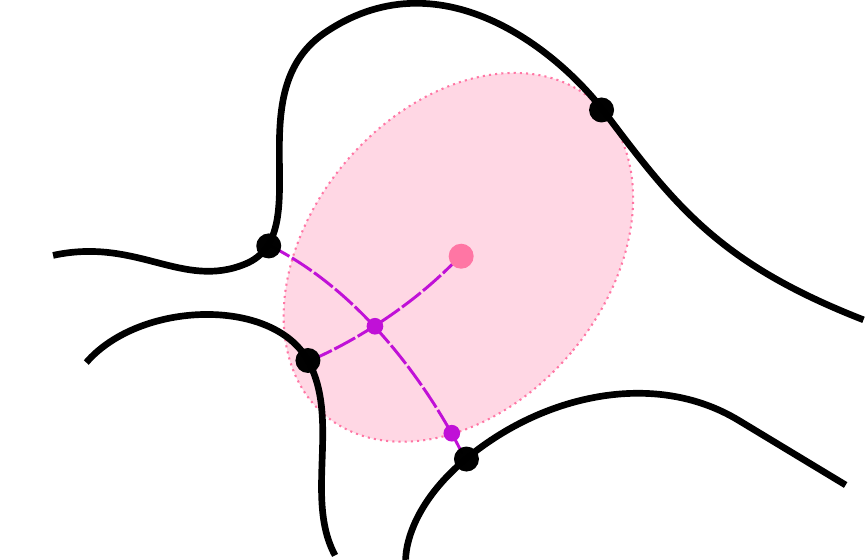}};
    \begin{scope}[x={(image.south east)},y={(image.north west)}]
        \node[FrenchPink] at (0.595,0.585) {$\gamma_s(\rho)$};
        \node[VividMulberry] at (0.475,0.4) {$y$};
        \node[VividMulberry] at (0.53,0.27) {$z$};
        \node[] at (0.32,0.33) {$p$};
        \node[] at (0.245,0.6) {$\Gamma(t_0)$};
        \node[] at (0.605,0.125) {$\Gamma(s_0)$};
        \node[] at (0.765,0.825) {$\Gamma(s)$};
    \end{scope}
\end{tikzpicture}
\vspace*{8pt}
\caption{Some relevant geodesic segments (in hatched purple) and points in the disk $D_\rho(s)$ (in pink).\label{fig:bords}}
\end{figure}

Therefore, we get
\begin{align*}
d_M(\Gamma(t_0),p)
&\leq d_M(\Gamma(t_0),y)+d_M(y,p) \\
&<d_M(\Gamma(t_0),y)+d_M(y,\Gamma(s_0)) \\
&=d_M(\Gamma(t_0),\Gamma(s_0)).
\end{align*}
By minimality of $s_0$, this implies that $p\notin \Gamma([0,\ell])$. But we also have that
\begin{align*}
d_M(\Gamma(t_0),\Gamma(s_0))<2\rho.
\end{align*}
Thus, if we also assume that $\rho_0\leq \frac{1}{2}d_M(x,\cup_i K_i)$, then $p$ will also not be in $u(\del S_r)-\Gamma([0,\ell])$. But this is of course a contradiction with the fact that $p\in u(\del S_r)$. \par

Therefore, $\gamma_s(\rho)$ and $\Gamma(s)$ must lie in different components of $D_\rho(s)-[\Gamma(t_0),\Gamma(s_0)]$. But this once again leads to a contradiction by the same logic as what we have just done: it suffices to replace every occurrence of $p$ by $\Gamma(s)$, and vice versa. Thus, it must be so that either $\Gamma(t_0)$ or $\Gamma(s_0)$ is in $D_\rho(s)$, and the second step of the proof is done. \par

For the third and final step of the proof, we suppose that $D_\rho(s)\cap \Gamma([t_0,s_0])\nsubseteq \del D_\rho(s)$ for all $s\in [t_0,s_0]$. By the previous step of the proof, if we get a contradiction, then we will have completed the proof. Fix $\epsilon\in (0,\rho_0-\rho)$. Suppose that there exist $s\neq t\in [t_0,s_0]$ such that $\Gamma(t)\in \operatorname{int}(D_\rho(s))$ and $\Gamma([t,s])\subseteq B_{\rho+\epsilon}(\gamma_t(\rho))$. By Lemma \ref{lem:disk_osculating}, the function $d(\tau)=d_M(\gamma_s(\rho),\Gamma(\tau))$ has a minimum at $s$. Since $d_M(\gamma_s(\rho),\Gamma(s))<\rho=d_M(\gamma_s(\rho),\Gamma(t))$, this thus implies that $d$ has a maximum at some point $s'\in (t,s)$. Therefore, by Lemma \ref{lem:disk_curvature}, we have that
\begin{align*}
\left|\frac{D}{ds}\dot{\Gamma}(s')\right|\geq\frac{\alpha}{\rho+\epsilon}>\Lambda,
\end{align*}
which is a contradiction. \par

\begin{figure}[ht]
\centering
\begin{tikzpicture}
    \node[anchor=south west,inner sep=0] (image) at (0,0) {\includegraphics[width=0.7\textwidth]{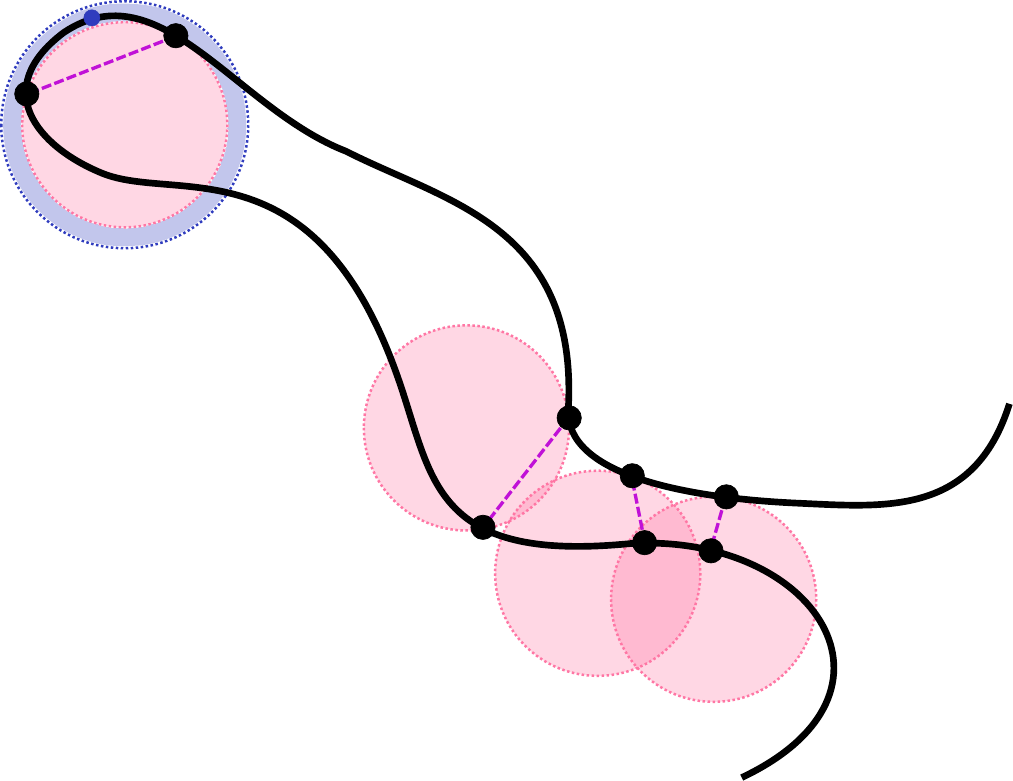}};
    \begin{scope}[x={(image.south east)},y={(image.north west)}]
        \node[DenimBlue] at (0.055,1.03) {$\left|\frac{D}{ds}\dot{\Gamma}\right|>\Lambda$};
        \node[] at (-0.05,0.88) {$\Gamma(s_n)$};
        \node[] at (0.25,0.98) {$\Gamma(t_n)$};
        \node[] at (0.43,0.28) {$\Gamma(s_2)$};
        \node[] at (0.62,0.49) {$\Gamma(t_2)$};
        \node[] at (0.595,0.25) {$\Gamma(s_1)$};
        \node[] at (0.665,0.425) {$\Gamma(t_1)$};
        \node[] at (0.7,0.235) {$\Gamma(s_0)$};
        \node[] at (0.765,0.4) {$\Gamma(t_0)$};
    \end{scope}
\end{tikzpicture}
\vspace*{8pt}
\caption{Multiple disks of the form $D_\rho(t_i)$ (in pink) with the inevitable point in $B_{\rho+\epsilon}(\gamma_{t_n}(\rho))$ breaking the curvature constraint (both in blue).\label{fig:neck}}
\end{figure}
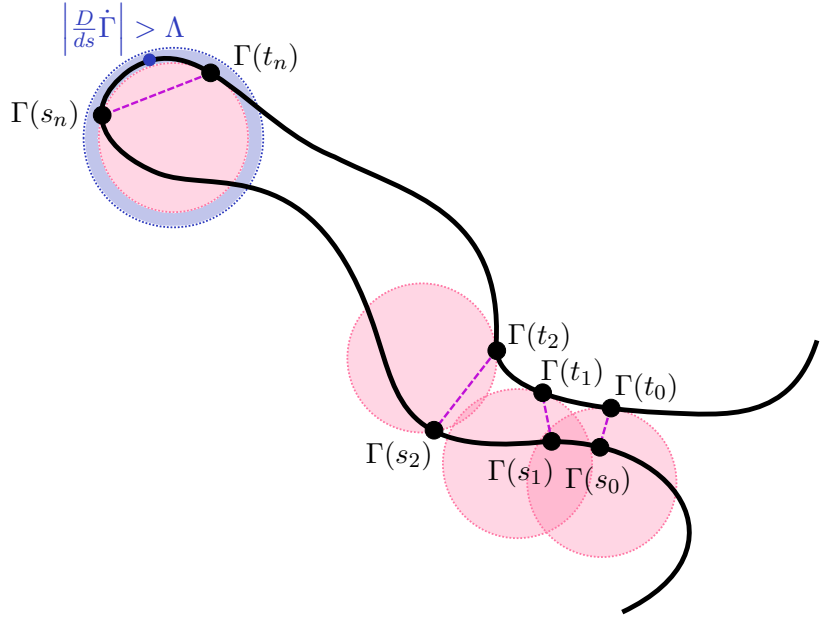

Summarizing what we have shown, for all $s\in [t_0,s_0]$, there exists $t\in [t_0,s_0]$ such that $\Gamma(t)\in \operatorname{int}(D_\rho(s))$. However, for all such $t$ and $s$, we have that $\Gamma([t,s])\nsubseteq B_{\rho+\epsilon}(\gamma_t(\rho))$. In particular, we have that $s_0-t_0\geq 2\epsilon$. Take $t_1:=t_0+\epsilon$. Note that, for any $s\in [t_0,t_1)$, we cannot have $\Gamma(s)\in \operatorname{int}(D_\rho(t_1))$, since $\Gamma([s,t_1])\subseteq B_{\rho+\epsilon}(\gamma_{t_1}(\rho))$. Therefore, we may take $s_1\in (t_1,s_0]$ such that
\begin{align*}
d_M(\Gamma(t_1),\Gamma(s_1))=\min\{d_M(\Gamma(t_1),\Gamma(s))| s\in [t_0,s_0], \Gamma(s)\in D_\rho(t_1)\}.
\end{align*}
Note that the second step of the proof stays true when we replace $[t_0,s_0]$ for $[t_1,s_1]$. In particular, we can similarly define $t_2:=t_1+\epsilon$ and $s_2\in [t_1,s_1]$. Continuing like this, we inductively define $t_{n+1}:=t_n+\epsilon$ and $s_{n+1}\in (t_{n+1},s_n]$ such that
\begin{align*}
d_M(\Gamma(t_{n+1}),\Gamma(s_{n+1}))=\min\{d_M(\Gamma(t_{n+1}),\Gamma(s))| s\in [t_n,s_n], \Gamma(s)\in D_\rho(t_{n+1})\}
\end{align*} 
for all $n\geq 0$. But of course, by construction,
\begin{align*}
2\epsilon\leq s_n-t_n\leq s_0-t_0-n\epsilon.
\end{align*}
Therefore, we run into a contradiction for $n$ large enough. This finally concludes the proof of Theorem \ref{thm:disk_dim2}.

\begin{rem} \label{rem:rho_0}
Summarizing all the choices that have been made for $\rho_0$, we can take
\begin{align*}
\rho_0=\min\left\{ \frac{r_0}{2},\frac{\pi}{2\sqrt{K_0}},\rho_1,\frac{\alpha}{\Lambda},\frac{1}{2}d_M(x,\cup_i K_i) \right\},
\end{align*}
where $\rho_1$ and $\alpha$ are the constants appearing in Lemma \ref{lem:disk_curvature}. As we will see below, $\rho_1\to \infty$ and $\alpha\to 1$ as $K_0\to 0$. Therefore, when $n=0$, we get $\rho_0\to \frac{1}{\Lambda}$ as $K_0\to 0$ and $r_0\to\infty$, thus indeed recovering the Moon in a puddle theorem.
\end{rem}

\subsection{Proof of technical results} \label{subsec:technical_dim2}
We now give the two proofs we had omitted in the previous subsection. \par

\begin{proof}[Proof of Lemma \ref{lem:disk_curvature}]
The proof is done in two steps.
\begin{enumerate}[label=(\arabic*)]
\item For a well-chosen variation of the minimal geodesic from $x$ to $\Gamma(s)$, we use the second variation formula for the energy functional to get that
\begin{align*}
\left|\frac{D}{ds}\dot{\Gamma}(s_0)\right|\geq \frac{I}{\rho}
\end{align*}
for some $I>0$ depending on the variation. \par

\item We use results on Jacobi fields to get a lower bound on $I$. \par
\end{enumerate}

We begin with the proof of the first step. For $\epsilon>0$ small enough, we consider the variation
\begin{center}
\begin{tikzcd}[row sep=0pt,column sep=1pc]
 h\colon [0,1]\times(-\epsilon,\epsilon) \arrow{r} & M \\
  {\hphantom{d\colon{}}} (t,s) \arrow[mapsto]{r} & \gamma_s(t):=\exp_x(t\tilde{\Gamma}(s_0+s)),
\end{tikzcd}
\end{center}
where $\exp_x(\tilde{\Gamma}(s_0+s))=\Gamma(s_0+s)$. In particular, $d(s)=\ell(\gamma_s)$. Therefore, since the length and energy functionals have the same critical points, and that $d$ achieve its maximum at $s_0$, we have that
\begin{align} \label{eqn:first-variation}
0=\frac{1}{2}E'(0) 
&= \left\langle\frac{\partial h}{\partial s},\frac{\partial h}{\partial t}\right\rangle(1,0)-\left\langle\frac{\partial h}{\partial s},\frac{\partial h}{\partial t}\right\rangle(0,0)
\end{align}
and
\begin{align} \label{eqn:second-variation}
0\geq \frac{1}{2}E''(0)=I\left(\frac{\partial h}{\partial s},\frac{\partial h}{\partial s}\right)+\left\langle\frac{D}{\partial s}\frac{\partial h}{\partial s},\frac{\partial h}{\partial t}\right\rangle(1,0)-\left\langle\frac{D}{\partial s}\frac{\partial h}{\partial s},\frac{\partial h}{\partial t}\right\rangle(0,0).
\end{align}
Here, $I$ denotes the index form of $\gamma:=\gamma_0$. For vector fields $V$ and $W$ along $\gamma$, it is defined as
\begin{align*}
I(V,W)=\int_0^1 \left(\left\langle\frac{DV}{dt},\frac{DW}{dt}\right\rangle-\left\langle R(\dot{\gamma},V)\dot{\gamma},W\right\rangle\right)dt.
\end{align*}
Note that the last term on the right-hand side of (\ref{eqn:second-variation}) is zero since $h(0,s)=x$ for all $s$. Furthermore, the middle term can be bounded from below:
\begin{align*}
\left\langle\frac{D}{\partial s}\frac{\partial h}{\partial s},\frac{\partial h}{\partial t}\right\rangle(1,0)
&= \left\langle\frac{D}{ds}c'(s_0),\dot{\gamma}(1)\right\rangle \\
&\geq -\left|\frac{D}{ds}c'(s_0)\right|\left|\dot{\gamma}(1)\right| \\
&= -\left|\frac{D}{ds}c'(s_0)\right| d(s_0) \\
&\geq -\left|\frac{D}{ds}c'(s_0)\right| \rho.
\end{align*}
Therefore, (\ref{eqn:second-variation}) turns into the desired bound of $\left|\frac{D}{ds}c'(s_0)\right|$ in terms of $I=I(\frac{\partial h}{\partial s},\frac{\partial h}{\partial s})$ and $\rho$. \par

We now turn to the second step of the proof. First of all, note that $J(t):=\frac{\del h}{\del s}(t,0)$ is a Jacobi field along $\gamma$. Therefore, since $\rho<r_\mathrm{inj}(M)$, we have that $J(t)\neq 0$ for all $t>0$. Furthermore, (\ref{eqn:first-variation}) implies that $J(1)$ and $\dot{\gamma}(1)$ are orthogonal, since the last term on the right-hand side is zero, as noted previously. The same is then true of $J(t)$ and $\dot{\gamma}(t)$ for all $t$ by standard results on Jacobi fields. Therefore, the index form simplifies slightly:
\begin{align*}
I=\int_0^1 \left(|\dot{J}|^2-K\left(\dot{\gamma},J\right)|J|^2|\dot{\gamma}|^2\right)dt \geq \int_0^1 \left(|\dot{J}|^2-K_0d(s_0)^2|J|^2\right)dt,
\end{align*}
where $\dot{J}:=\frac{D}{dt}J$. \par

However, we have that
\begin{align*}
|\dot{J}||J|\geq |\langle \dot{J},J \rangle|=|\dot{\gamma}|\left|(\operatorname{Hess}\rho)(J,J)\right|=d(s_0)\left|(\operatorname{Hess}\rho)(J,J)\right|,
\end{align*}
where $\rho$ is the distance function from $x$. Therefore, by the Hessian comparison theorem (see for example \cite{GreeneWu1979}) and the Rauch comparison theorem, we get that
\begin{align*}
I
&\geq \int_0^1 \left(\cot^2(\sqrt{K_0}d(s_0)t)-1\right)K_0d(s_0)^2|J(t)|^2dt \\
&\geq |\dot{\tilde{\Gamma}}(s_0)|^2 \int_0^1 \left(\cot^2(\sqrt{K_0}d(s_0)t)-1\right)\sin^2(\sqrt{K_0}d(s_0)t)dt \\
&= \frac{\sin(2\sqrt{K_0}d(s_0))}{2\sqrt{K_0}d(s_0)}|\dot{\tilde{\Gamma}}(s_0)|^2.
\end{align*}
But, using the Rauch comparison theorem again, we get
\begin{align*}
1=|\dot{\Gamma}(s_0)|=|(d\exp_x)_{\tilde{\Gamma}(s_0)}(\dot{\tilde{\Gamma}}(s_0))|\leq \frac{\sinh(\sqrt{K_0}d(s_0))}{\sqrt{K_0}d(s_0)}|\dot{\tilde{\Gamma}}(s_0)|.
\end{align*}
Therefore, this finally implies that
\begin{align*}
I\geq \frac{\sqrt{K_0}d(s_0)\sin(2\sqrt{K_0}d(s_0))}{2\sinh^2(\sqrt{K_0}d(s_0))}.
\end{align*}
However, the function $\tau\mapsto (\tau\sin(2\tau))/(2\sinh^2(\tau))$ is positive and decreasing on $(0,\frac{\pi}{2})$. Therefore, if we take
\begin{align*}
\rho_1:=\min\left\{r_0,\frac{\pi}{2\sqrt{K_0}}\right\}-\epsilon
\end{align*}
for any $\epsilon>0$ small enough, than we will have
\begin{align*}
I\geq \frac{\sqrt{K_0}\rho_1\sin(2\sqrt{K_0}\rho_1)}{2\sinh^2(\sqrt{K_0}\rho_1)}=:\alpha>0,
\end{align*}
which concludes the proof of the lemma.
\end{proof}

\begin{proof}[Proof of Lemma \ref{lem:disk_osculating}]
We suppose the contrary and get a contradiction as follows.
\begin{enumerate}[label=(\arabic*)]
\item Since $D_\rho(s)\cap \Gamma([s-\epsilon,s+\epsilon])\nsubseteq \del D_\rho(s)$ for all $\epsilon>0$, there must exist a decreasing sequence $\{\epsilon_n\}_{n\geq 1}\subseteq \R_{>0}$ converging to 0 such that $\Gamma(s\pm \epsilon_n)\in D_\rho(s)$ for all $n\geq 1$. By passing to a subsequence and changing the orientation of the parametrization $\Gamma$ if needed, we may assume that $\Gamma(s + \epsilon_n)\in D_\rho(s)$ for all $n\geq 1$. \par

\item Since $\Gamma$ is parametrized by arclength, we have the inclusion $\Gamma((s,s+\epsilon_n))\subseteq B_{\rho+\epsilon_n}(\gamma_s(\rho))$. We recall that $\gamma_s(\rho)=\exp_{\Gamma(s)}(\rho N(s))$, and $N$ is the unit-length vector field along $\Gamma$ which is orthogonal to $\dot{\Gamma}$ and pointing toward the interior of the topological disk $\overline{u(S_r)}$. \par

\item Suppose that for all $n\geq 1$, there exists $s_n\in (s,s+\epsilon_n)$ such that 
\begin{align*}
d_M(\gamma_s(\rho),\Gamma(s_n))\geq\rho.
\end{align*}
Note that $s_n\to s$, because $\epsilon_n\to 0$. Since $\Gamma|_{[s,s+\epsilon_n]}$ respects the hypotheses of Lemma \ref{lem:disk_curvature}, we get
\begin{align*}
\left|\frac{D}{ds}\dot{\Gamma}(s)\right|
&= \lim_{n\to\infty}\left|\frac{D}{ds}\dot{\Gamma}(s_n)\right| \\
&\geq \lim_{n\to\infty}\frac{\alpha}{\rho+\epsilon_n} \\
&= \frac{\alpha}{\rho} \\
&> \Lambda,
\end{align*}
which is of course a contradiction. Therefore, it must be that $\Gamma([s,s+\epsilon_n])\subseteq D_\rho(s)$ for all $n\geq 1$. \par

\item Note that $s$ must be a critical point of the function $d$ of Lemma \ref{lem:disk_curvature} for $x=\gamma_s(\rho)$. Indeed, those correspond to the critical points of the energy functional along the variation $h(\tau,t)=\exp_{\gamma_s(\rho)}(\tau \tilde{\Gamma}(s+t))$, where $\exp_{\gamma_s(\rho)}(\tilde{\Gamma}(s+t))=\Gamma(s+t)$. But, for such a variation,
\begin{align*}
\frac{1}{2}E'(0)
&= \left\langle \frac{\del h}{\del t}, \frac{\del h}{\del \tau} \right\rangle (1,0) \\
&= \left\langle \dot{\Gamma}(s),\left.\frac{d}{d\tau}\exp_{\gamma_s(\rho)}(\tau\tilde{\Gamma}(s))\right|_{\tau=1} \right\rangle \\
&= \frac{1}{\rho} \left\langle \dot{\Gamma}(s),-N(s) \right\rangle \\
&= 0.
\end{align*}
Indeed, the path $\tau\mapsto\exp_{\gamma_s(\rho)}(\tau\tilde{\Gamma}(s))$ is the unique minimizing geodesic of speed $|\tilde{\Gamma}(s)|=\rho$ from $\gamma_s(\rho)$ to $\Gamma(s)$. This is just $\gamma_s$ parametrized in the opposite orientation and with a different speed. \par

Furthermore, the fact that $\Gamma([s,s+\epsilon_n])\subseteq D_\rho(s)$ for all $n\geq 1$ implies that $d$ must have nonpositive second derivative at $s$. This in turn implies that $E''(0)\leq 0$. Therefore, all the proof of Lemma \ref{lem:disk_curvature} still works, and we get a contradiction:
\begin{align*}
\left|\frac{D}{ds}\dot{\Gamma}(s)\right|\geq\frac{\alpha}{\rho}>\Lambda.
\end{align*}
\end{enumerate}
\end{proof}

\section{Badly-behaved sequences} \label{sec:bad}
We conclude this paper with examples of sequences of Lagrangian submanifolds. These sequences show that bounds of curvature type are needed to ensure convergence in the Hausdorff metric. \par

Consider the sequence of Hamiltonians $\{H_n(x,y):=\frac{1}{n}\sin(nx)\}_{n\geq 1}$ on the 2-torus $\mathds{T}^2=\R^2/2\pi\Z^2$. We equip the torus with the standard symplectic form $\omega_0$, the standard complex structure $J_0$, and the flat metric $g_0=g_{J_0}$. A quick calculation shows that the induced Hamiltonian flow of $H_n$ is given by
\begin{align*}
\phi^t_n(x,y)=(x,y+t\cos(nx)),
\end{align*}
for all $t\geq 0$, $n\geq 1$ and $(x,y)\in \mathds{T}^2$. We then set
\begin{align*}
L_0:=\{y=0\} \quad\text{and}\quad L_n:=\phi^1_n(L_0)=\{y=\cos(nx)\}.
\end{align*}

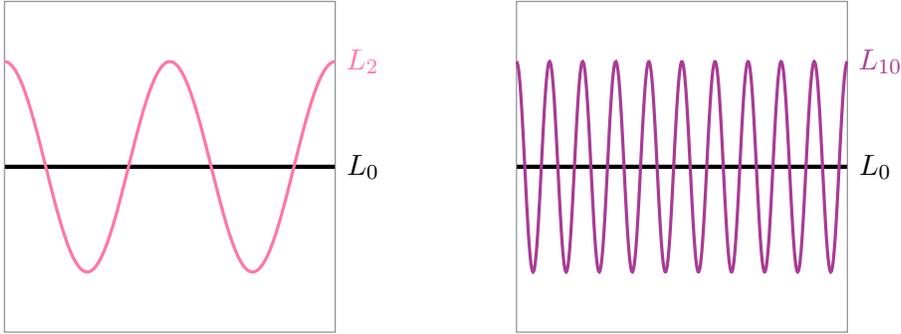
\begin{figure}[ht]
\centering
\begin{tikzpicture}[scale=0.7]
  \draw[-,ultra thick] (-pi, 0) -- (pi, 0) node[right] {$L_0$};
  \draw[gray,thin] (-pi,-pi) rectangle (pi,pi);
  \draw[domain=-pi:pi,samples=300,FrenchPink,very thick] plot(\x,{2*cos(2*\x r)}) node[right] {$L_2$};
\end{tikzpicture}
\qquad\qquad
\begin{tikzpicture}[scale=0.7]
  \draw[-,ultra thick] (-pi, 0) -- (pi, 0) node[right] {$L_0$};
  \draw[gray,thin] (-pi,-pi) rectangle (pi,pi);
  \draw[domain=-pi:pi,samples=500,Mulberry,very thick] plot(\x,{2*cos(10*\x r)}) node[right] {$L_{10}$};
\end{tikzpicture}
\caption{$L_n$ as $n$ gets larger}
\label{fig:bad-sequence_hofer}
\end{figure}

Note that $d_H(L_0,L_n)=\frac{2}{n}$, and thus $L_n$ tends to $L_0$ in the Lagrangian Hofer metric. Indeed, $\frac{2}{n}$ is an upper bound of $d_H(L_0,L_n)$, since it is the oscillation of $H_n$. On the other hand, by Corollary 3.13 of \cite{BarraudCornea2006}, $d_H(L_0,L_n)$ is bounded from below by the area of a strip. However, it is clear that such a strip must have area at least $\frac{2}{n}$. \par

On the other hand, as Figure \ref{fig:bad-sequence_hofer} suggests, $L_n$ tends to the full band $\{-1\leq y\leq 1\}$ in the Hausdorff metric $\delta_H$ induced by the flat metric. Therefore, even though the sequence converges in both metric, the limits are quite different. Actually, we have that $\delta_H(L_0,L_n)\equiv 1$. \par

The conjecture does not apply to this sequence of Lagrangian submanifolds, as the norm of the second fundamental form is clearly unbounded. Actually, a quick calculation gives that 
\begin{align*}
||B_{L_n}||=\max_{p\in K_n}|\kappa(p)|=n,
\end{align*}
where $\kappa$ denotes the geodesic curvature. \par

Note that this example can easily be generalized to higher dimensional tori. Likewise, by multiplying $H_n$ by a cutoff function, this example can be transposed to any symplectic manifold using a Darboux chart. Finally, since $d_H$ bounds from above every other metric mentioned in the introduction, this problem applies to every known metric for which Theorem \ref{thm:main} holds. In other words, this is a universal example. \par

However, it is possible to get a sequence of Lagrangian submanifolds with exploding curvature, but where limits in a $J$-adapted metric and the Hausdorff metric coincide. For example, one can do the analogous construction as above, but with Hamiltonians $G_n:=\frac{1}{\sqrt{n}}H_n$. Indeed, the associated Hamiltonian flows are
\begin{align*}
\psi^t_n(x,y)=\left(x,y+\frac{t}{\sqrt{n}}\cos(nx)\right),
\end{align*}
and we get the Lagrangian submanifolds
\begin{align*}
K_n:=\psi^1_n(L_0)=\left\{y=\frac{1}{\sqrt{n}}\cos(nx)\right\}.
\end{align*}
By arguments similar to the above, one gets $d_H(L_0,K_n)=2n^{-3/2}$ and $\delta_H(L_0,K_n)=n^{-1/2}$. Therefore, $K_n$ tends to $L_0$ in both the Lagrangian Hofer and the Hausdorff metric. However, it is easy to calculate that the maximum of the curvature of $K_n$ is given by
\begin{align*}
\max_{p\in K_n}|\kappa(p)|=\sqrt{n}.
\end{align*}
This, of course, tends to infinity as $n$ tends to infinity. \par

\begin{rem}
Note that, in the sequence $\{L_n\}$ above, not only does $||B_{L_n}||$ tends to infinity, but it is also impossible to uniformly tame the Lagrangian submanifolds in the sequence. Indeed, the distance between two successive zeroes of $y=\cos(nx)$ is $\frac{\pi}{n}$ in $M$, but is at least 2 in $L_n$. Therefore,
\begin{align*}
\lim_{n\to\infty}\inf_{x\neq y\in L_n}\frac{d_M(x,y)}{\min\{1,d_L(x,y)\}}=0.
\end{align*}
We expect this phenomenon to be general: when $M$ is simply connected, control over the second fundamental form should give enough control over tameness for the proof of Theorem \ref{thm:main-precise} to still work. It would then be possible to pass to the universal cover to get the desired result, just as we have done in the proof of Theorem \ref{thm:main-precise_dim2}. \par

In order to clarify our intuition, let us note that $L$ being $\epsilon$-tame is equivalent to the following condition: for all $x\in L$, and for all $y\in B_\epsilon(x)\cap L$, we have that
\begin{align} \label{eqn:t'1}
d_M(x,y)\leq \frac{1}{\epsilon}d_L(x,y).
\end{align}
It is quite clear that any type of bound on curvature cannot stop $\epsilon$ from being arbitrarily small at some point $x\in L$. On the hand, as Figure \ref{fig:neck} from the previous section suggests, having this condition at every $x\in L$ should force a certain lower bound on $||B_L||$. Therefore, for any $L\in\mathscr{L}^\star_\Lambda(M)$, there should be some $x\in L$ where the optimal epsilon appearing in (\ref{eqn:t'1}) is bounded from below by some constant $e=e(K_0,r_0,\Lambda)>0$. We could then apply Proposition \ref{prop:monotonicity} on some appropriate metric ball centered at this $x$; the size of the ball would only depend on $K_0$, $r_0$, and $\Lambda$.
\end{rem}

\clearpage
\bibliographystyle{alpha}
\bibliography{convergence6}

\textsc{Department of Mathematics and Statistics, Universit\'e de Montr\'eal, C.P. 6128 succ. Centre-ville, Montreal, QC, H3C 3J7, Canada} \\
\textit{E-mail}: \href{mailto:jean-philippe.chasse@umontreal.ca}{jean-philippe.chasse@umontreal.ca}

\end{document}